\documentclass[a4paper,12pt]{amsart}
\clearpage{}\usepackage{adjustbox}
\usepackage{amsfonts}
\usepackage{amsmath}
\usepackage{amssymb}
\usepackage{amsthm}
\usepackage{bbm}
\usepackage{cases}
\usepackage{enumerate}
\usepackage{enumitem}
\usepackage{epigraph}
\usepackage[margin=1in]{geometry}
\usepackage{graphicx}
\usepackage{hyperref}
\usepackage[capitalise]{cleveref}
\usepackage{mathrsfs}
\usepackage{microtype}
\usepackage{subcaption}
\usepackage{thmtools}
\usepackage{tikz-cd} 
\usepackage{tikz}
\usepackage[textsize=scriptsize]{todonotes}
\usepackage{verbatim}
\usepackage{xcolor}

\usetikzlibrary{decorations.pathreplacing}
\usetikzlibrary{decorations.markings}
\usetikzlibrary{spy}

\tikzset{->-/.style={decoration={
  markings,
  mark=at position .5 with {\arrow{>}}},postaction={decorate}}}
\tikzset{-<-/.style={decoration={
  markings,
  mark=at position .5 with {\arrow{<}}},postaction={decorate}}}

\def\drawsemicircle(#1,#2)(#3,#4)(#5){\pgfmathsetmacro\tA{#3-#1}
\pgfmathsetmacro\tL{#4-#2}
\pgfmathsetmacro\tH{sqrt(\tA*\tA + \tL*\tL)}
\pgfmathsetmacro\sA{asin(\tL/\tH)}
\draw[#5] (#3,#4) arc
[ start angle=\sA, end angle=\sA + 180,
  radius = \tH/2
]}

\def\drawHsemicircle(#1,#2)(#3)(#4)(#5){\draw[#3] (#2,0) arc
[ start angle = 0, end angle = 180, radius = (#2-#1)/2]
node[#4]{#5}
}

\presetkeys{todonotes}{inline}{}

\makeatletter
\providecommand\@dotsep{5}
\makeatother

\makeatother

\DeclareMathOperator{\Hom}{Hom}
\DeclareMathOperator{\HN}{HN}
\DeclareMathOperator{\Id}{Id}

\DeclareMathOperator{\Ob}{Ob}
\DeclareMathOperator{\PSL}{PSL}

\DeclareMathOperator{\Stab}{Stab}

\DeclareMathOperator{\id}{Id}

\DeclareMathOperator{\occ}{occ}

\DeclareMathOperator{\std}{std}

\def\1{\mathbbm{1}}
\def\a{\alpha}

\def\b{\beta}

\def\C{\mathcal{C}}
\def\CC{\mathbb{C}}

\def\dd{\partial}

\def\D{\mathcal{D}}

\def\DGM{{\: \mathsf{DGM}}}
\def\e{\epsilon}

\def\f{\phi}

\def\g{\gamma}
\def\G{\mathcal{G}}

\def\H{\mathcal{H}}
\def\HH{\mathbb{H}}

\def\hrt{\heartsuit}
\def\heart{\heartsuit}

\def\I{\mathcal{I}}

\def\k{\Bbbk}

\def\La{\Lambda}

\def\N{\mathbb{N}}
\def\o{\omega}
\def\ol{\overline}

\def\OO{\mathbb{O}}

\def\p{\pi}

\def\P{\mathbb{P}}
\def\cP{\mathcal{P}}

\def\Q{\mathbb{Q}}

\def\R{\mathcal{R}}
\def\RR{\mathbb{R}}

\def\s{\sigma}
\def\sm{\setminus}

\def\S {\mathcal{S}}
\def\bS{\mathbb{S}}

\def\t{\tau}
\def\T{\mathcal{T}}

\def\v{\vert}

\def\w{\omega}
\def\wh{\widehat}

\def\Z{\mathbb{Z}}
\def\cZ{Z}

\makeatletter

\theoremstyle{plain}
\newtheorem{theorem}{Theorem}[section]
\newtheorem{prop}[theorem]{Proposition}
\crefname{prop}{Proposition}{Propositions}
\newtheorem{lemma}[theorem]{Lemma}
\newtheorem{cor}[theorem]{Corollary}
\newtheorem{conjecture}[theorem]{Conjecture}
\theoremstyle{definition}
\newtheorem{example}[theorem]{Example}
\newtheorem{mydef}[theorem]{Definition}
\newtheorem{definition}[theorem]{Definition}
\newtheorem{remark}[theorem]{Remark}

\makeatletter
\let\@wraptoccontribs\wraptoccontribs
\makeatother
\clearpage{}

\title{$q$-deformed rational numbers and the 2-Calabi--Yau category of type $A_{2}$ }
\author{Asilata Bapat}
\address[Asilata Bapat]{Mathematical Sciences Institute, Australian National University}
\email{asilata.bapat@anu.edu.au}
\author{Louis Becker}
\address[Louis Becker]{Australian National University}
\email{louis.becker@anu.edu.au}
\author{Anthony M. Licata}
\address[Anthony M. Licata]{Mathematical Sciences Institute, Australian National University}
\email{anthony.licata@anu.edu.au}

\begin{document}
\maketitle

We describe a family of compactifications of the space of Bridgeland stability conditions of any triangulated category following earlier work by Bapat, Deopurkar, and Licata.
We particularly consider the case of the 2-Calabi--Yau category of the $A_2$ quiver.
The compactification is the closure of an embedding (depending on $q$) of the stability space into an infinite-dimensional projective space.

In the $A_2$ case, the three-strand braid group $B_3$ acts on this closure.
We describe two distinguished braid group orbits in the boundary, points of which can be identified with certain rational functions in $q$.
Points in one of the orbits are exactly the $q$-deformed rational numbers recently introduced by Morier-Genoud and Ovsienko, while the other orbit gives a new $q$-deformation of the rational numbers.
Specialising $q$ to a positive real number, we obtain a complete description of the boundary of the compactification.

\section{Introduction}

\subsection{Deformed rational numbers}
The theory of $q$-deformations of rational numbers has an extremely recent history.
Deformed rational numbers were originally introduced in~\cite{mor.ovs:20} and developed in~\cite{mor.ovs:19*1}.
The definition considered in these papers is via deformations of continued fraction expansions of rational numbers.
The resulting $q$-rationals enjoy several pleasant properties, and the authors discuss connections to a wide variety of classical topics including the Farey triangulation, cluster algebras, and the Jones polynomial.
\par
In this paper we present an enhancement of the notion of $q$-deformed rationals, from which the $q$-rationals of~\cite{mor.ovs:20} can be easily read off.
For any fixed positive $q$ and a rational number $r/s$, we describe a pair of real numbers, denoted $[r/s]^{\flat}$ and $[r/s]^{\sharp}$.
The number $[r/s]^{\sharp}$ is precisely the $q$-deformation of $r/s$ considered in~\cite{mor.ovs:20}.
We argue that $[r/s]^{\flat}$ can also reasonably be thought of as a $q$-deformation of $r/s$.

For example, the number $[1/0]^{\sharp}$, which is the $q$-deformation of $\infty$ in~\cite{mor.ovs:20}, is just $\infty$ itself.
On the other hand, $[1/0]^{\flat} = 1/(1-q)$.
For other rationals $r/s$, the two $q$-deformations can both be described via the action of a group $\PSL_{2,q}(\Z) \subset \PSL_2(\RR)$, which acts on $\RR \cup \{\infty\}$ by fractional linear transformations.
We define the groups $\PSL_{2,q}(\Z)$ and give the construction of the two $q$-deformations in parallel in~\cref{sec:deformation}.

\subsection{A compactification of the space of stability conditions}
The motivation for our two constructions of $q$-deformed rational numbers comes from homological algebra.
To explain this, in ~\cref{sec:compactification}, we associate to each positive real $q$ a compactification of the space of Bridgeland stability conditions on $\C_{2}$, the 2-Calabi--Yau (2-CY) triangulated category for the $A_2$ quiver.
This compactification is a generalisation of a construction of~\cite{bap.deo.lic:20}, which corresponds to the case $q=1$, and proceeds by embedding the space of stability conditions into an infinite dimensional projective space, and taking the closure there.
We set $M_q$ to be the image of the space of stability conditions under this embedding,  $\overline{M_q}$ to be its closure, and $\partial \overline{M_q}$ to be $\overline{M_q} \backslash M_q$.
In the construction, the ambient projective space is independent of $q$, but the embedding itself depends on $q$.

When $q=1$, the spaces involved can be identified with basic objects of hyperbolic geometry: the space of stability conditions $\Stab(\C_{2})/\CC$ can be identified with the upper half plane $\mathbb{H}^2$ (homeomorphic to the open disk), and the compactification is then the closed disk $\mathbb{H}^2 \cup \mathbb{R} \cup \{\infty\}$.
The points of the boundary don't correspond to stability conditions, but they can be regarded as translation-invariant metrics on the category $\C_2$.
Moreover, there is a canonical bijection between the spherical objects of $\C_2$ up to shift and the rationals $\mathbb{Q}\cup \{\infty\} \subset \RR\cup \{\infty\}$.
It is a theorem of Rouquier--Zimmermann \cite{rou.zim:03} that this bijection intertwines the braid group action on the spherical objects with the $\PSL_2(\mathbb{Z})$ action on $\mathbb{Q}\cup \{\infty\}$.
Details about this construction may be found in \cite{bap.deo.lic:20} and references therein.

One of the main goals of this paper is to establish analogues of this result for the case $q\ne 1$.
When we vary $q$, the space $M_q$ remains an open disk, but something interesting happens at the boundary: each spherical object in the category corresponds not to a rational point, but rather to an entire interval in $\dd\ol{M}_{q}$.
We elaborate on this statement in~\cref{subsec:dynamics}, by associating two projectivised linear functionals to each spherical object $X$.
These functionals, called $\ol{\hom}_{q}(X,\cdotp)$ and $\occ_{q}(X,\cdotp)$, each map the set of spherical objects to $\RR$.
The former sends $Y$ to the Laurent polynomial whose $q^{k}$-coefficient is the number of degree-$k$ morphisms from $X$ to $Y$. The definition of $\occ_{q}(X, \cdotp)$ is a little more involved, but is also intrinsically well-motivated.
At $q=1$, these two functionals agree and lie on the boundary, coinciding with the rational point in $\RR \cup \{\infty\}$ corresponding to the spherical object $X$.
When $q \neq 1$, they are distinct functionals, and the entire closed interval of convex combinations of the two functionals lies on the boundary.
We call this interval $I_{q,X}$.
We summarise these statements in the following theorem.
\begin{theorem}\label{introthm1}\
  \begin{enumerate}
  \item For all $q\in(0,\infty)$, the space of stability conditions maps homeomorphically onto ${M}_{q}$, which is in turn homeomorphic to an open disk.
  \item The boundary of $\overline{M_q}$ is $\dd\ol{M}_{q}$, and is homeomorphic to $\mathbb{R} \cup \{\infty\}$.
  \item The union of the intervals $I_{q,X}$, as $X$ ranges over the spherical objects of $\C_2$, is a dense subset of the boundary $\dd\ol{M}_q$.
  \end{enumerate}
\end{theorem}
Further, we conjecture that $\ol{M}_{q}$ is homeomorphic to a closed disk.
We prove this theorem in~\cref{sec:compactification}.
The main tool used in the proof is a finite state automaton known as a Harder--Narasimhan automaton or HN automaton, introduced in ~\cite{bap.deo.lic:20}.
This automaton controls how the HN filtration of an object in the $A_2$ category transforms when acted on by a braid.
(This automaton is also of use when $q=1$, where is allows for a simpler proof of Rouquier--Zimmermann's original result.)

We expect the role of HN automata in the study of triangulated autoequivalence groups will be somewhat analogous to the role of Thurston's train track automata in the study of mapping class groups. For 2-CY catgories, such automata have been constructed in type $A_2$ and type $\widehat{A}_1$ in ~\cite{bap.deo.lic:20}, and in the PhD thesis of Edmund Heng in all other rank 2 Coxeter types. Constructing HN automata in higher rank examples and exploiting them to study stability conditions and autoequivalence groups is an important open problem. We give the definition of HN automata in \cref{app:stability-and-automata}.

\subsection{Relationship to the $q$-rational numbers}
Recall that at $q = 1$, the correspondence of spherical objects to rational numbers intertwines the braid group action on the spherical objects with the $\PSL_2(\mathbb{Z})$ action on $\mathbb{R}\cup\{\infty\}$.
The bulk of~\cref{sec:homological} is devoted to the $q$-analogue of this correspondence, where the group $\PSL_2(\mathbb{Z})$ is replaced by the group $\PSL_{2,q}(\mathbb{Z})$.
We prove the following result (see~\cref{qrz1,qrz2} for the precise statements).
\begin{theorem}
  Consider assignments
  \[X \mapsto (-q)^{-\epsilon}\frac{\occ_q(P_2,X)}{\occ_q(P_1,X)} \quad\text{and}\quad X \mapsto q^{-1}(-q)^{\epsilon}\frac{\ol{\hom}_q(X,P_2)}{\ol{\hom}_q(X,P_1)},\]
  where $\epsilon$ is either $0$ or $1$ depending on $X$.
  Each of these assignments intertwines the $B_3$ action on the spherical objects of $\C_2$ with the $\PSL_{2,q}(\mathbb{Z})$ action on $\mathbb{R}\cup \{\infty\}$.
\end{theorem}
In fact, the first of the two assignments above recovers the $q$-rational numbers of \cite{mor.ovs:20,mor.ovs:19*1}.
\begin{theorem}\label{introtemp}
  Let $X$ be a spherical object of $\C_2$ corresponding to the rational number $r/s$.
  The ratio
  \[(-q)^{-\epsilon}\frac{\occ_q(P_2,X)}{\occ_q(P_1,X)}\]
  is the $q$-deformation $[r/s]_{q}^{\sharp}$ in the sense of Morier-Genoud--Ovsienko \cite{mor.ovs:20,mor.ovs:19*1}.
\end{theorem}
This result suggests that the corresponding ratio of $\ol{\hom}_q$ functionals
\[q^{-1}(-q)^{\epsilon}\frac{\ol{\hom}_{q}(X,P_{2})}{\ol{\hom}_{q}(X,P_{1})}\]
should be thought of as another $q$-deformation of the corresponding rational number; it is in fact the alternative $q$-deformation $[r/s]_{q}^{\flat}$ mentioned above.
In~\cref{sec:deformation} we give an elementary parallel account of both of these $q$-deformations, and~\cref{qrz1,qrz2} give precise relationships between the $q$-deformations and homological algebra.

\subsection{The Jones polynomial of rational knots}
In \cite[Appendix A]{mor.ovs:20}, Morier-Genoud and Ovsienko relate the $q$-deformed rational numbers to the Jones polynomial of rational (two-bridge) knots. In particular, they prove that the normalised Jones polynomial is equal to $q\R(q)+(1-q)\S(q)$, where $[r/s]^{\sharp}_{q} = \R(q)/\S(q)$.
In our \cref{appendixa}, we prove an analogous result involving our new $q$-deformed rational number $[r/s]^{\flat}_{q}$: the numerator of $[r/s]^{\flat}_{q}$ is equal to the Jones polynomial of the corresponding rational knot (up to the signs of the coefficients).
\section{$q$-deformed rational numbers}\label{sec:deformation}
In this section we present two different $q$-deformations of the set $\Q\cup \{\infty\}$, one of which is exactly the $q$-deformed rational numbers of \cite{mor.ovs:20}.
We begin by defining the two $q$-deformations via continued fraction expansions.
Although the definitions may appear unmotivated at first, the remainder of the section provides several viewpoints towards the two deformations that also serve as motivation.
Throughout this paper, a rational number is considered to be any element of $\Q\cup \{\infty\}$.
Also throughout this paper, we treat the case $0<q \leq 1$. The case $1\leq q< \infty$ is symmetric, and, up to sign, all of our results are preserved under this symmetry.
\subsection{Left and right $q$-deformed rational numbers}\label{subsec:qrationals-def}
Recall that every non-zero rational number $r/s$ has a unique even continued fraction expression $[a_{1}, ...,a_{2n}]$, where
$$\frac{r}{s}=a_{1}+\cfrac{1}{a_{2}+\cfrac{1}{a_{3}+\cfrac{1}{\cfrac{\ddots}{a_{2n-1}+\cfrac{1}{a_{2n}}}}}},$$
and either $a_{1} \in \N$ and $a_{2},..,a_{2n} \in \N\sm \{0\}$, or $-a_{1} \in \N$ and $-a_{2},...,-a_{2n} \in \N\sm \{0\}$.
The continued fraction expansion of 0 is $[-1,1]$. The continued fraction expansion of $\infty$ is the empty expansion $[\,]$.

\begin{remark}
  The convention for continued fractions that we have chosen here is slightly nonstandard.
  However, we have the following pleasing dichotomy under this convention:
  \begin{enumerate}
  \item strictly positive rational numbers only contain non-negative numbers in their continued fraction expansions;
  \item strictly negative rational numbers only contain non-positive numbers in their continued fraction expansions.
  \end{enumerate}
The points 0 and $\infty$ are both negative and positive, and thus exceptional. The choice of which corresponds to the empty expansion is equivalent to the choice between odd and even continued fraction expansions. The odd continued fraction expansion of 0 is $[\,]$, and that of $\infty$ is $[0,1,-1]$. For even continued fractions, $\infty$ is the `basepoint' in \cref{continuedmatrices}; for odd continued fractions, it is 0.
\end{remark}

The continued fraction expansion is related to the action on $\mathbb{Q}\cup \{\infty\}$ by the modular group $\PSL_{2}(\Z)$.
Recall that $\PSL_2(\RR)$ is the quotient of the group of $2\times 2$ real invertible matrices by scalar matrices.
The group $\PSL_2(\Z)$ is the subgroup of $\PSL_2(\RR)$ generated by
\begin{align*} \s_{1}:=\begin{bmatrix} 1& -1 \\ 0 & 1\end{bmatrix}, \quad \s_{2}:= \begin{bmatrix} 1 & 0 \\ 1 & 1 \end{bmatrix}. \end{align*}
The group $\PSL_2(\RR)$ acts on each of the sets
\[\RR \cup \{\infty\} \subset \CC \cup \{\infty\}\]
by fractional linear transformations:
\begin{equation}\label{frlitr}\begin{bmatrix} 
a & b \\ c & d
\end{bmatrix} \cdot z=\frac{az+b}{cz+d}.
\end{equation}
The modular group $\PSL_2(\Z)$ also preserves, and is transitive on, $\Q \cup \{\infty\}$.
The following is easy to check.
\begin{prop}\label{continuedmatrices}
  Let \(r/s\) be a rational number.
  Suppose that \(r/s\) has an even continued fraction expansion \([a_1, \ldots, a_{2n}]\).
  Then
  \[\frac{r}{s} = \sigma_{1}^{-a_1}\sigma_{2}^{a_2}\sigma_{1}^{-a_3}\cdots \sigma_{1}^{-a_{2n-1}}\sigma_{2}^{a_{2n}}(\infty),\]
  under the usual action of \(\PSL_2(\mathbb{Z})\) on \(\mathbb{Q} \cup\{\infty\}\).
\end{prop}
In order to obtain $q$-deformations of the rationals, we first define a family of discrete subgroups of $\PSL_2(\mathbb{R})$ indexed by $q$, which deform the modular group $\PSL_2(\Z)$.
\begin{definition}
  Set $\PSL_{2,q}(\mathbb{Z})$ to be the subgroup of $\PSL_2(\mathbb{R})$ generated by
\[ \s_{1,q}:=
  \begin{bmatrix} q^{-1} & - q^{-1} \\ 0 & 1 \end{bmatrix},
  \quad \s_{2,q}:=\begin{bmatrix} 1 & 0 \\ 1 & q^{-1} \end{bmatrix}.\]
\end{definition}
At $q = 1$, the group $\PSL_{2,q}(\Z)$ is simply $\PSL_2(\Z)$.
Recall the three-strand Artin braid group:
$$B_{3}:=\langle \s_{1}, \s_{2} \mid \s_{1}\s_{2}\s_{1}=\s_{2}\s_{1}\s_{2} \rangle.$$
The following is easy to check.
\begin{prop}
  The assignment
  \[\sigma_1 \mapsto \sigma_{1,q}, \quad \sigma_2 \mapsto \sigma_{2,q}\]
  gives a homomorphism from $B_3$ to $\PSL_{2,q}(\mathbb{Z})$.
\end{prop}
Our definitions of the $q$-deformed rational numbers are motivated by~\cref{continuedmatrices}, and use the homomorphism from $B_3 \to \PSL_{2,q}(\Z)$ described above.
We first set up the following notation.
\begin{definition}
  Given an even continued fraction expansion
  \[\mathbf{a} = [a_1, \ldots, a_{2n}]\]
  as above,
  define the braid corresponding to \(\mathbf{a}\) as the following element of \(B_3\):
  \[\beta_{\mathbf{a}} = \sigma_1^{-a_1}\sigma_2^{a_2}\sigma_1^{-a_3}\cdots \sigma_1^{-a_{2n-1}} \sigma_2^{a_{2n}}.\]
  For any \(q \in \mathbb{R}\), denote by \(\beta_{\mathbf{a},q}\) the image of \(\beta_{\mathbf{a}}\) in \(\PSL_{2,q}(\mathbb{Z})\).
\end{definition}
Now we are in a position to define the $q$-deformed rational numbers.
\begin{definition}[Left and right $q$-deformed rationals]
  Let $r/s \in \Q\cup \{\infty\}$ with continued fraction expansion $\mathbf{a}$.
  \begin{enumerate}
  \item The \emph{left} $q$-deformed rational number corresponding to $r/s$ is
    \[\left[\frac{r}{s}\right]^{\flat}_q = \beta_{\mathbf{a},q}\left(\frac{1}{1-q}\right).\]
  \item The \emph{right} $q$-deformed rational number corresponding to $r/s$ is
    \[\left[\frac{r}{s}\right]^{\sharp}_q = \beta_{\mathbf{a},q}\left(\infty\right).\]
  \end{enumerate}
\end{definition}
Note that at $q = 1$, the left and right deformed rationals coincide, and by~\cref{continuedmatrices}, are both equal to $r/s$.
By the results of~\cite[\textsection 2.6]{mor.ovs:20}, it is clear that our right $q$-deformed rational numbers are the same as the $q$-rational numbers considered in~\cite{mor.ovs:20}.
In the next section we motivate our definition of left $q$-deformed rational numbers.
\subsection{Motivation via the $q$-Farey tessellation}\label{motivation}
In this section we recall a $q$-deformation of the classical Farey tessellation of the hyperbolic plane, originally considered by Morier-Genoud--Ovsienko in \cite{mor.ovs:20}.
We first recall the classical Farey tessellation.
The hyperbolic plane has an upper half-plane model:
$$\HH:=\{z \in \CC \mid \Im(z)>0\},$$
which embeds naturally into $\CC \cup \{\infty\}$.
The group $\PSL_2(\RR)$, which acts on $\CC \cup \{\infty\}$ by fractional linear transformations, preserves the hyperbolic plane $\HH$, and acts on it by hyperbolic isometries.

The hyperbolic plane admits a tessellation by ideal triangles according to the following iterative construction. We begin with an `initial geodesic' $E$, namely the straight vertical line from $0$ to $\infty$. Starting with $E$,
we complete every geodesic in the tessellation to a triangle according to the following rule, called the \emph{Farey rule}. By induction, every geodesic in the tessellation has rational vertices $r/s$ and $r'/s'$ (the integers $r,s,r',s'$ are uniquely determined by construction). We add a geodesic from $r/s$ to $(r+r')/(s+s')$ and a geodesic from $(r+r')/(s+s')$ to $r'/s'$ (see \cref{fig:stfatr}).
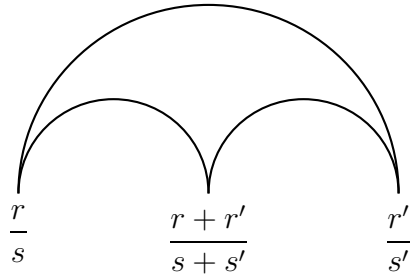
\begin{figure}[h]
\tikzset{every picture/.style={line width=0.75pt}} \begin{tikzpicture}
    \node(A) at (0, 0) [below] {$\displaystyle\frac{r}{s}$};
    \node(B) at (2.5, 0) [below] {$\displaystyle\frac{r+r'}{s+s'}$};
    \node(C) at (5, 0) [below] {$\displaystyle\frac{r'}{s'}$};        
    \drawsemicircle(0,0)(5,0)(thick);
    \drawsemicircle(0,0)(2.5,0)(thick);
    \drawsemicircle(2.5,0)(5.0,0)(thick);
\end{tikzpicture}
\caption{A standard Farey triangle.}\label{fig:stfatr}
\end{figure}
The `positive half' of the hyperbolic plane, those $z \in \HH$ such that $0\leq \Re(z) < \infty$, is covered by applying the rule to the initial geodesic $E$ with its vertices viewed as the fractions $0/1$ and $1/0$. The `negative half' is obtained by applying the rule to $E$ with its vertices viewed as the fractions $0/1$ and $-1/0$.
We call the resulting tessellation the Farey tessellation of the hyperbolic plane. The set of vertices of the Farey tessellation is exactly $\Q \cup \{\infty\}$. Moreover, the Farey tessellation is preserved by the $\PSL_{2}(\Z)$-action described above (that is, the $\PSL_{2}(\Z)$-action sends triangles to triangles), and the $\PSL_{2}(\Z)$-orbit of any Farey triangle is the entire hyperbolic plane.
\par
Morier-Genoud and Ovsienko introduced a \emph{$q$-deformed Farey rule}, which generates a tessellation of a subset of $\HH$.
As with the regular Farey rule, we begin with a geodesic $E$ between $0$ and $\infty$, but it is now assigned a label $q^{-1}$. We then extend every geodesic in the tessellation, with vertices $\R(q)/\S(q)$ and $\R'(q)/\S'(q)$ (the vertices are now rational functions evaluated at $q$), to a triangle according to \cref{qfatr} (see~\cite[Definition 2.9]{mor.ovs:20}).
\begin{figure}[h]

\tikzset{every picture/.style={line width=0.75pt}} 

\begin{tikzpicture}[x=0.75pt,y=0.75pt,yscale=-1,xscale=1]

\draw    (219,131) .. controls (247.5,7) and (399.5,5) .. (424,131) ;
\draw    (219,131) .. controls (241.7,86.65) and (277.7,90.65) .. (290.5,128) ;
\draw    (290.5,128) .. controls (313.7,55.65) and (391.7,43.65) .. (424,131) ;

\draw (213,133) node [anchor=north west][inner sep=0.75pt]   [align=left] {$\displaystyle \frac{\mathcal{R}( q)}{\mathcal{S}( q)}$};
\draw (263.5,132) node [anchor=north west][inner sep=0.75pt]   [align=left] {$\displaystyle \frac{\mathcal{R}( q) +q^{k}\mathcal{R} '( q)}{\mathcal{S}( q) +q^{k}\mathcal{S} '( q)}$};
\draw (314,13) node [anchor=north west][inner sep=0.75pt]   [align=left] {$\displaystyle q^{k-1}$};
\draw (253,80) node [anchor=north west][inner sep=0.75pt]   [align=left] {$\displaystyle 1$};
\draw (344,47) node [anchor=north west][inner sep=0.75pt]   [align=left] {$\displaystyle q^{k}$};
\draw (408,135) node [anchor=north west][inner sep=0.75pt]   [align=left] {$\displaystyle \frac{\mathcal{R} '( q)}{\mathcal{S} '( q)}$};

\end{tikzpicture}

\caption{A $q$-deformed Farey triangle.}\label{qfatr}
\end{figure}
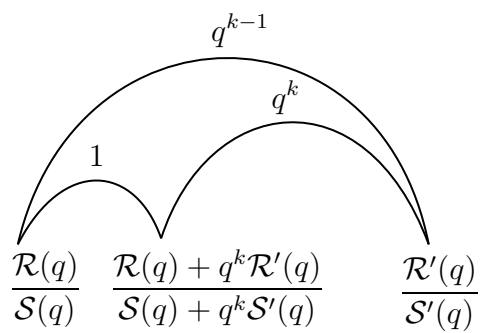
For the positive half, view $E$ as being the top edge in \cref{qfatr} connecting the fractions $0/1$ and $1/0$, and for the negative half, as connecting the fractions $0/q$ and $-1/0$.
The bottom left edge of any triangle is always labelled $1$, and we thus iteratively deduce labels for all other edges.
The set of vertices of the $q$-deformed Farey tessellations is exactly the set of \emph{right} $q$-deformed rational numbers $[r/s]^{\sharp}_q$.
\cref{fig:tessellations-at-various-q} depicts what the $q$-deformed Farey tessellations look like at two different values of $q$.
\begin{figure}[h]
  \centering
  \begin{subfigure}[h]{0.4\linewidth}
\centering
\begin{tikzpicture}
\drawHsemicircle(0.0,5.0)(line width = 0.4pt)()();
\drawHsemicircle(0.0,1.1538461538461537)(line width = 0.3pt)()();
\drawHsemicircle(1.1538461538461537,5.0)(line width = 0.3pt)()();
\drawHsemicircle(0.0,0.32374100719424453)(line width = 0.2pt)()();
\drawHsemicircle(0.32374100719424453,1.1538461538461537)(line width = 0.2pt)()();
\drawHsemicircle(1.1538461538461537,1.4028776978417266)(line width = 0.2pt)()();
\drawHsemicircle(1.4028776978417266,5.0)(line width = 0.2pt)()();
\drawHsemicircle(0.0,0.09527170077628794)(line width = 0.1pt)()();
\drawHsemicircle(0.09527170077628794,0.32374100719424453)(line width = 0.1pt)()();
\drawHsemicircle(0.32374100719424453,0.3881884538818845)(line width = 0.1pt)()();
\drawHsemicircle(0.3881884538818845,1.1538461538461537)(line width = 0.1pt)()();
\drawHsemicircle(1.1538461538461537,1.2143273150844496)(line width = 0.1pt)()();
\drawHsemicircle(1.2143273150844496,1.4028776978417266)(line width = 0.1pt)()();
\drawHsemicircle(1.4028776978417266,1.4714184897671136)(line width = 0.1pt)()();
\drawHsemicircle(1.4714184897671136,5.0)(line width = 0.1pt)()();
\end{tikzpicture}
\caption{$q = 0.3$}
  \end{subfigure}
  \begin{subfigure}[h]{0.4\linewidth}
\centering
    \begin{tikzpicture}
      \drawHsemicircle(0.0,5.0)(line width = 0.4pt)()();
\drawHsemicircle(0.0,2.058823529411765)(line width = 0.3pt)()();
\drawHsemicircle(2.058823529411765,5.0)(line width = 0.3pt)()();
\drawHsemicircle(0.0,1.1187214611872145)(line width = 0.2pt)()();
\drawHsemicircle(1.1187214611872145,2.058823529411765)(line width = 0.2pt)()();
\drawHsemicircle(2.058823529411765,2.7168949771689497)(line width = 0.2pt)()();
\drawHsemicircle(2.7168949771689497,5.0)(line width = 0.2pt)()();
\drawHsemicircle(0.0,0.6770627714172917)(line width = 0.1pt)()();
\drawHsemicircle(0.6770627714172917,1.1187214611872145)(line width = 0.1pt)()();
\drawHsemicircle(1.1187214611872145,1.3777704267284154)(line width = 0.1pt)()();
\drawHsemicircle(1.3777704267284154,2.058823529411765)(line width = 0.1pt)()();
\drawHsemicircle(2.058823529411765,2.370862975564491)(line width = 0.1pt)()();
\drawHsemicircle(2.370862975564491,2.7168949771689497)(line width = 0.1pt)()();
\drawHsemicircle(2.7168949771689497,3.0260560600078956)(line width = 0.1pt)()();
\drawHsemicircle(3.0260560600078956,5.0)(line width = 0.1pt)()();
\end{tikzpicture}
\caption{$q = 0.7$}
  \end{subfigure}
  \caption{The deformed Farey tessellations at two values of $q$.}
  \label{fig:tessellations-at-various-q}
\end{figure}
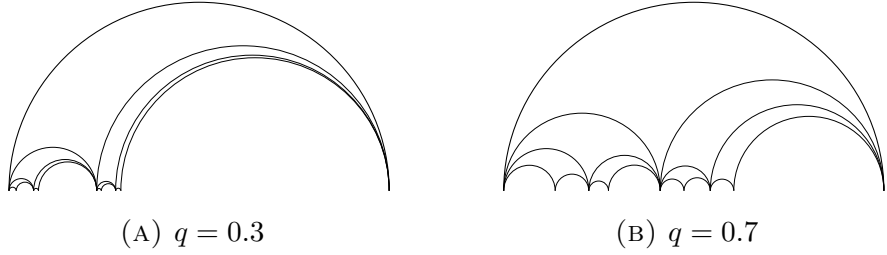

The action of $\PSL_{2,q}(\Z)$ on $\HH$ preserves $q$-deformed Farey triangles (since $\PSL_{2,q}(\Z)$ is a subgroup of $\PSL_{2}(\RR)$, which preserves geodesics). Moreover, $\PSL_{2,q}(\Z)$ acts on the $q$-deformed Farey tessellation exactly as $\PSL_{2}(\Z)$ acts on the classical Farey tessellation. More formally, let $\mathbb{T}$ denote the set of triangles in the Farey tessellation, and let $\mathbb{T}_{q}$ denote the set of $q$-deformed triangles.
Since the $\PSL_{2}(\Z)$-action preserves Farey triangles, there is an induced $\PSL_{2}(\Z)$-action on $\mathbb{T}$, and similarly there is an induced $\PSL_{2,q}(\Z)$-action on $\mathbb{T}_{q}$.
Let $\psi \colon \PSL_{2,q}(\Z) \to \PSL_2(\Z)$ be the group homomorphism given by setting $q = 1$, and let $\f \colon \mathbb{T}_q \to \mathbb{T}$ be induced from the map $\R(q)/\S(q) \mapsto \R(1)/\S(1)$ on the vertices.
Then for any $\b \in \PSL_{2, q}(\Z)$, we have the following commutative square, as can be proven by comparing the actions on the vertices of the respective tessellations.
\[
\begin{tikzcd}
\mathbb{T}_{q} \arrow[d, "\beta"] \arrow[r, "\phi"] & \mathbb{T} \arrow[d, "\psi(\beta)"] \\
\mathbb{T}_{q} \arrow[r, "\phi"]                    & \mathbb{T}                         
\end{tikzcd}
\]

We can now describe how left $q$-rationals arise in the $q$-deformed Farey tessellation. Consider any rational $r/s$. In the standard Farey tessellation, $r/s$ is the middle vertex of a unique triangle, and the left (resp. right) vertex of an infinite sequence of triangles. In the notation of \cref{fig:stdfarey}, $r/s$ is the right vertex of $T_{0},T_{1},...$, and the left vertex of $T_{0}^{'},T_{1}^{'},...$ .

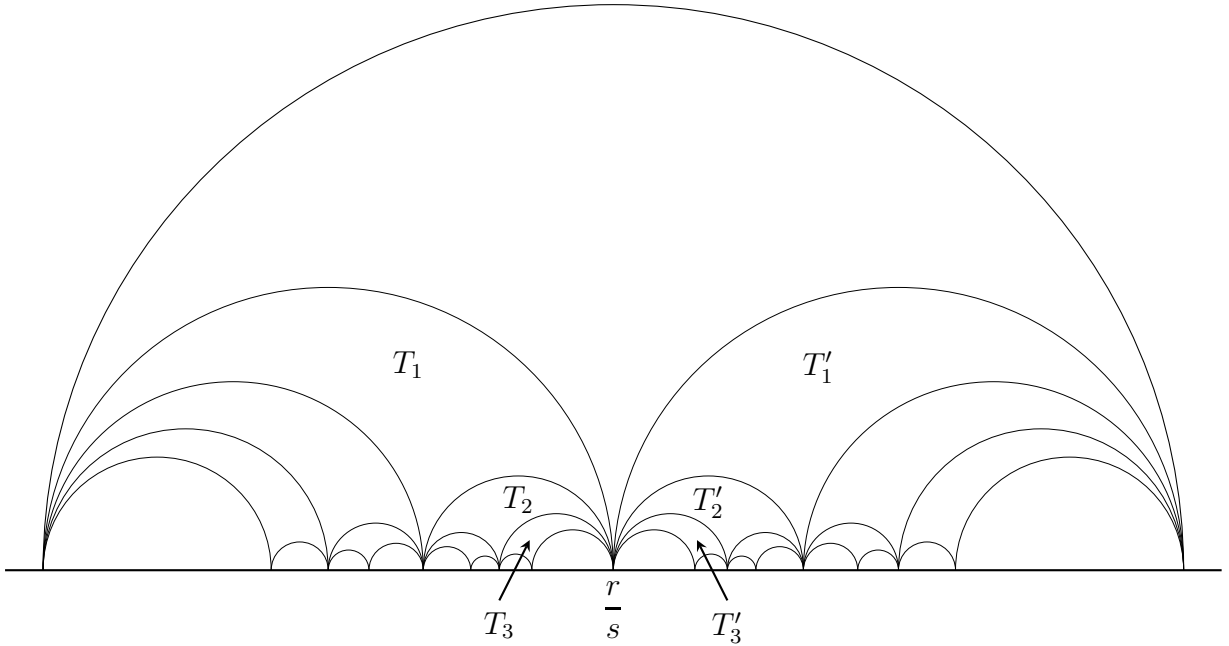
\begin{figure}[h]
  \centering
  \begin{tikzpicture}

\drawHsemicircle(0.0,15.0)(line width = 0.4pt)()();
\drawHsemicircle(0.0,7.5)(line width = 0.32pt)(midway, below right=1cm)($T_1$);
\drawHsemicircle(7.5,15.0)(line width = 0.32pt)(midway, below left=1cm)($T_1'$);
\drawHsemicircle(0.0,5.0)(line width = 0.24pt)()();
\drawHsemicircle(5.0,7.5)(line width = 0.24pt)(midway, below)($T_2$);
\drawHsemicircle(7.5,10.0)(line width = 0.24pt)(midway, below)($T_2'$);
\drawHsemicircle(10.0,15.0)(line width = 0.24pt)()();
\drawHsemicircle(0.0,3.75)(line width = 0.16pt)()();
\drawHsemicircle(3.75,5.0)(line width = 0.16pt)()();
\drawHsemicircle(5.0,6.0)(line width = 0.16pt)()();
\drawHsemicircle(6.0,7.5)(line width = 0.16pt)()();
\draw[thick, ->, >=stealth] (6,-0.4) node [below]{$T_3$} -- (6.4,0.4);
\drawHsemicircle(7.5,9.0)(line width = 0.16pt)()();
\draw[thick, ->, >=stealth] (9,-0.4) node [below]{$T_3'$} -- (8.6,0.4);
\drawHsemicircle(9.0,10.0)(line width = 0.16pt)()();
\drawHsemicircle(10.0,11.25)(line width = 0.16pt)()();
\drawHsemicircle(11.25,15.0)(line width = 0.16pt)()();
\drawHsemicircle(0.0,3.0)(line width = 0.08pt)()();
\drawHsemicircle(3.0,3.75)(line width = 0.08pt)()();
\drawHsemicircle(3.75,4.285714285714286)(line width = 0.08pt)()();
\drawHsemicircle(4.285714285714286,5.0)(line width = 0.08pt)()();
\drawHsemicircle(5.0,5.625)(line width = 0.08pt)()();
\drawHsemicircle(5.625,6.0)(line width = 0.08pt)()();
\drawHsemicircle(6.0,6.428571428571429)(line width = 0.08pt)()();
\drawHsemicircle(6.428571428571429,7.5)(line width = 0.08pt)()();
\drawHsemicircle(7.5,8.571428571428571)(line width = 0.08pt)()();
\drawHsemicircle(8.571428571428571,9.0)(line width = 0.08pt)()();
\drawHsemicircle(9.0,9.375)(line width = 0.08pt)()();
\drawHsemicircle(9.375,10.0)(line width = 0.08pt)()();
\drawHsemicircle(10.0,10.714285714285714)(line width = 0.08pt)()();
\drawHsemicircle(10.714285714285714,11.25)(line width = 0.08pt)()();
\drawHsemicircle(11.25,12.0)(line width = 0.08pt)()();
\drawHsemicircle(12.0,15.0)(line width = 0.08pt)()();
\draw[thick] (-0.5,0) -- (15.5,0);
\node at (7.5,0) [below] {$\displaystyle \frac{r}{s}$};
  \end{tikzpicture}
  \caption{A fraction in the standard Farey tessellation.}\label{fig:stdfarey}
  \label{fig:stdfarey2}
\end{figure}

Assume for convenience that $0<r/s<\infty$, and let $T_{A}$ denote the triangle with vertices $0,1,\infty$.
If the continued fraction expansion of $r/s$ is $\mathbf{a}$, then 
\[T_{0}=\b_{\mathbf{a}}T_{A}, \quad T_{n}=(\b_{\mathbf{a}}\s_{1}^{-1}\b_{\mathbf{a}}^{-1})^{n}T_{0}.\]
Similarly, letting $\b_{\mathbf{a}}^{'}=\b_{\mathbf{a}}\s_{2}^{-1}\s_{1}^{-1}$, we have
\[T_{0}^{'}=\b_{\mathbf{a}}^{'}T_{A}, \quad T_{n}^{'}=(\b_{\mathbf{a}}^{'}\s_{2}(\b_{\mathbf{a}}^{'})^{-1})^{n}T_{0}.\]

As suggested by \cref{fig:stdfarey}, for any point $x \in T_{0}$,
$$\lim_{n \to \infty}(\b_{\mathbf{a}}\s_{1}^{-1}\b_{\mathbf{a}}^{-1})^{n}x=\frac{r}{s},$$
and for any point $x' \in T_{0}^{'}$,
$$\lim_{n \to \infty}(\b_{\mathbf{a}}^{'}\s_{2}(\b_{\mathbf{a}}^{'})^{-1})^{n}x'=\frac{r}{s}.$$
So we can view the rational point $r/s$ in the boundary as the limit of the sequences of Farey triangles $(T_{n})_{n \in \N}, (T_{n}^{'})_{n \in \N}$.

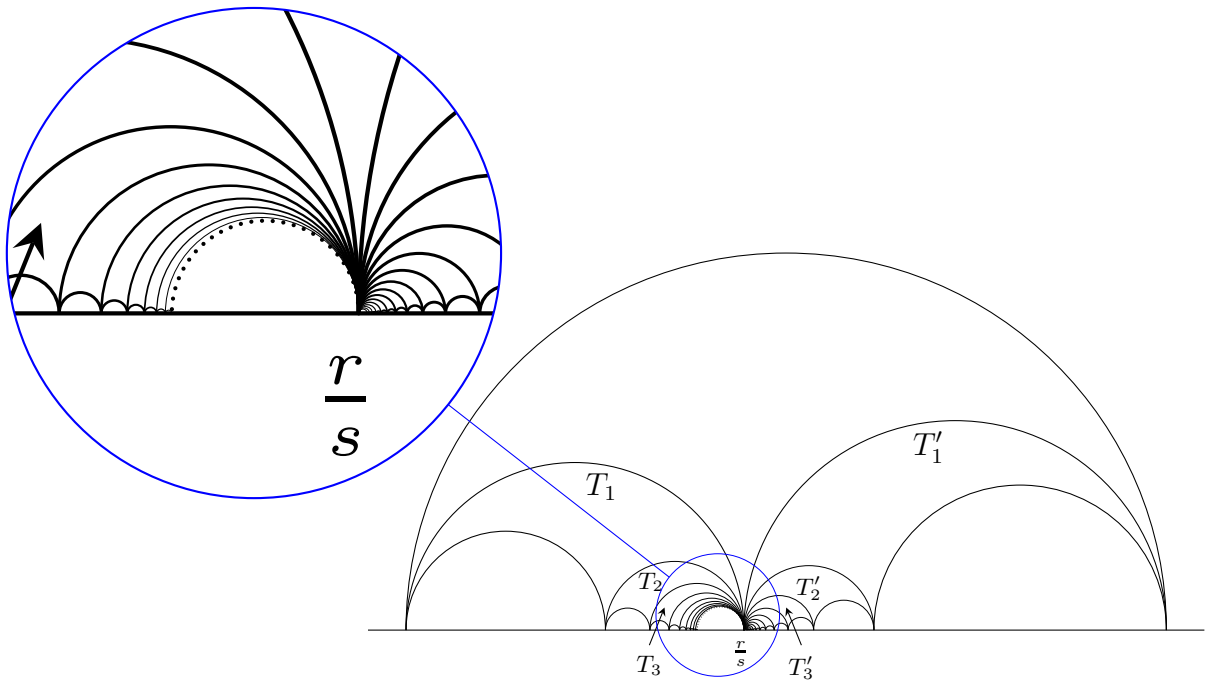
\begin{figure}[h]
\begin{center}
  \begin{tikzpicture}[spy using outlines]
    \drawHsemicircle(0,10)()()();
\drawHsemicircle(0.0,4.444444444444445)(line width = 0.4pt)(midway, below right)($T_1$);
\drawHsemicircle(0.0,2.6229508196721314)(line width = 0.4pt)()();
\drawHsemicircle(2.6229508196721314,4.444444444444445)(line width = 0.36pt)(font=\scriptsize, midway, below left)($T_2$);
\drawHsemicircle(2.6229508196721314,3.2071269487750556)(line width = 0.36pt)()();
\drawHsemicircle(3.2071269487750556,4.444444444444445)(line width = 0.32pt)()();
\draw[->, >=stealth] (3.2, -0.2) node[below, font=\scriptsize]{$T_3$}  -- (3.4, 0.3);
\drawHsemicircle(3.2071269487750556,3.459766040411202)(line width = 0.32pt)()();
\drawHsemicircle(3.459766040411202,4.444444444444445)(line width = 0.28pt)()();
\drawHsemicircle(3.459766040411202,3.5980254738253405)(line width = 0.28pt)()();
\drawHsemicircle(3.5980254738253405,4.444444444444445)(line width = 0.24pt)()();
\drawHsemicircle(3.5980254738253405,3.6835011669826105)(line width = 0.24pt)()();
\drawHsemicircle(3.6835011669826105,4.444444444444445)(line width = 0.2pt)()();
\drawHsemicircle(3.6835011669826105,3.7403810864024307)(line width = 0.2pt)()();
\drawHsemicircle(3.7403810864024307,4.444444444444445)(line width = 0.16pt)()();
\drawHsemicircle(3.7403810864024307,3.780107993683086)(line width = 0.16pt)()();
\drawHsemicircle(3.780107993683086,4.444444444444445)(line width = 0.12pt)()();
\drawHsemicircle(3.780107993683086,3.8088010328397925)(line width = 0.12pt)()();
\drawHsemicircle(3.8088010328397925,4.444444444444445)(line width = 0.08pt)()();
\drawHsemicircle(3.8088010328397925,3.8300305201200535)(line width = 0.08pt)()();
\drawHsemicircle(3.8300305201200535,4.444444444444445)(line width = 0.4pt, line cap = round, dash pattern=on 0pt off 1pt)()();

\drawHsemicircle(4.444444444444445,10.0)(line width = 0.4pt)(midway, below left)($T_1'$);
\drawHsemicircle(6.153846153846153,10.0)(line width = 0.4pt)()();
\drawHsemicircle(4.444444444444445,6.153846153846153)(line width = 0.37pt)(font=\scriptsize, midway, below)($T_2'$);
\drawHsemicircle(5.360824742268042,6.153846153846153)(line width = 0.37pt)()();
\draw[->, >=stealth] (5.2,-0.2) node[below, font=\scriptsize]{$T_3'$} -- (5,0.3);
\drawHsemicircle(4.444444444444445,5.360824742268042)(line width = 0.35pt)()();
\drawHsemicircle(5.02446982055465,5.360824742268042)(line width = 0.35pt)()();
\drawHsemicircle(4.444444444444445,5.02446982055465)(line width = 0.32pt)()();
\drawHsemicircle(4.842046407604138,5.02446982055465)(line width = 0.32pt)()();
\drawHsemicircle(4.444444444444445,4.842046407604138)(line width = 0.29pt)()();
\drawHsemicircle(4.729844980685296,4.842046407604138)(line width = 0.29pt)()();
\drawHsemicircle(4.444444444444445,4.729844980685296)(line width = 0.27pt)()();
\drawHsemicircle(4.655423384666736,4.729844980685296)(line width = 0.27pt)()();
\drawHsemicircle(4.444444444444445,4.655423384666736)(line width = 0.24pt)()();
\drawHsemicircle(4.603559460399648,4.655423384666736)(line width = 0.24pt)()();
\drawHsemicircle(4.444444444444445,4.603559460399648)(line width = 0.21pt)()();
\drawHsemicircle(4.566158878120386,4.603559460399648)(line width = 0.21pt)()();
\drawHsemicircle(4.444444444444445,4.566158878120386)(line width = 0.19pt)()();
\drawHsemicircle(4.538518347651663,4.566158878120386)(line width = 0.19pt)()();
\drawHsemicircle(4.444444444444445,4.538518347651663)(line width = 0.16pt)()();
\drawHsemicircle(4.517718336630378,4.538518347651663)(line width = 0.16pt)()();
\drawHsemicircle(4.444444444444445,4.517718336630378)(line width = 0.13pt)()();
\drawHsemicircle(4.50185209145278,4.517718336630378)(line width = 0.13pt)()();
\drawHsemicircle(4.444444444444445,4.50185209145278)(line width = 0.11pt)()();
\drawHsemicircle(4.489623596945505,4.50185209145278)(line width = 0.11pt)()();
\drawHsemicircle(4.444444444444445,4.489623596945505)(line width = 0.08pt)()();
\drawHsemicircle(4.480123524670582,4.489623596945505)(line width = 0.08pt)()();
\drawHsemicircle(4.444444444444445,4.480123524670582)(line width = 0.05pt)()();
\drawHsemicircle(4.472697393327169,4.480123524670582)(line width = 0.05pt)()();
\drawHsemicircle(4.444444444444445,4.472697393327169)(line width = 0.03pt)()();
\drawHsemicircle(4.466864375947158,4.472697393327169)(line width = 0.03pt)()();
\draw[thin] (-0.5,0) -- (10.5,0);
\node at (4.4,0) [below, font=\scriptsize] {$\frac{r}{s}$};

\spy[blue, draw, circle, width=6.5cm, magnification = 4, connect spies] on (4.1, 0.2) in node at (-2,5);
\end{tikzpicture}
\end{center}
\caption{The sequence of triangles $T_n$ converging to the $q$-deformed rational curve of $r/s$, shown as a dotted line.}\label{fig:qfarey}
\end{figure}

When we $q$-deform the Farey tessellation, we obtain corresponding sequences of triangles $(T_{q,n})_{n \in \N}, (T_{q,n}^{'})_{n\in \N}$.
As we will partially prove in \cref{subsec:limits-q-rationals}, based on arguments from~\cite{mor.ovs:19*1}, the triangles $T_{q,n}^{'}$ once again limit to a point, but the triangles $T_{q,n}$ limit to a geodesic, with endpoints on the boundary (see \cref{fig:qfarey}).
We call this geodesic the \emph{$q$-deformed rational curve} corresponding to $r/s$.
The left $q$-rational corresponding to $r/s$ is the left endpoint of a $q$-deformed rational curve, and the right $q$-rational corresponding to $r/s$ is the right endpoint of the same curve.
We denote by $\Q_q$ the union of all of the $q$-deformed rational curves, together with their endpoints.
This motivates our definition of left and right $q$-rationals.
Namely, if $T_{n} \xrightarrow{n \to \infty} r/s$, then the corresponding $q$-deformed sequence of triangles $T_{n,q}$ limits to the geodesic with endpoints $[r/s]_{q}^{\flat}$ and $[r/s]_{q}^{\sharp}$.

\subsection{Alternate descriptions of the deformed rational numbers}
The left and right $q$-deformed rational numbers admit several elegant intrinsic definitions.
Recall the classical $q$-deformation of the integers, which we call the \emph{right} $q$-deformation:
\[[n]^{\sharp}_{q}:=\frac{1-q^{n}}{1-q} = 1 + q + \cdots + q^{n-1}.\]
We introduce a corresponding \emph{left} $q$-deformation of $n$:
\[[n]_{q}^{\flat}:=\frac{1-q^{n-1}+q^{n}-q^{n+1}}{1-q} = 1 + q + \cdots + q^{n-2} + q^n.\]

The following proposition gives equivalent definitions of the two $q$-deformations of the rational numbers, and is easy to check.
In fact, this is how the (right) $q$-deformed rational numbers were originally defined (see~\cite[Definition 1.1]{mor.ovs:20}).
\begin{prop}\label{qratdef}
  Let $r/s$ be a finite rational number with even continued fraction expansion $[a_{1},...,a_{2n}]$.
  \begin{enumerate}
  \item  The corresponding right $q$-deformed rational has the following formula.
    \[\Big[\cfrac{r}{s}\Big]^{\sharp}_{q} =
      [a_{1}]^{\sharp}_{q}+\cfrac{q^{a_{1}}}{[a_{2}]^{\sharp}_{q^{-1}}+\cfrac{q^{-a_{2}}}{[a_{3}]^{\sharp}_{q}+\cfrac{q^{a_{3}}}{[a_{4}]^{\sharp}_{q^{-1}}+\cfrac{q^{-a_{4}}}{\cfrac{\ddots}{[a_{2n-1}]^{\sharp}_{q}+\cfrac{q^{a_{2n-1}}}{[a_{2n}]^{\sharp}_{q^{-1}}}}}}}}\]

  \item The corresponding left $q$-deformed rational has the following formula.
    \[\Big[\cfrac{r}{s}\Big]_{q}^{\flat}
      =[a_{1}]^{\sharp}_{q}+\cfrac{q^{a_{1}}}{[a_{2}]^{\sharp}_{q^{-1}}+\cfrac{q^{-a_{2}}}{\cfrac{\ddots}{[a_{2n-1}]^{\sharp}_{q}+\cfrac{q^{a_{2n-1}}}{[a_{2n}]_{q^{-1}}^{\flat}}}}}\]
  \end{enumerate}
\end{prop}

\begin{mydef}\label{qrat}
  Following \cite{mor.ovs:20}, we fix the following notation for the rational functions associated to the left and right $q$-deformed rational numbers:
  \[\left[\frac{r}{s}\right]_{q}^{\flat}=\frac{\R^{\flat}(q)}{\S^{\flat}(q)},\quad
    \left[\frac{r}{s}\right]_{q}^{\sharp}=\frac{\R^{\sharp}(q)}{\S^{\sharp}(q)},\]
  where the polynomials in the numerator and denominator are normalised so that if $r/s>0$ (resp. $r/s<0$), the denominator has lowest (resp. highest) degree term equal to 1.
  If $r/s=0$, we set
  \[\R^{\flat}(q) = 1-q^{-1},\quad \S^{\flat}(q) = 1, \quad
    \R^{\sharp}(q)=0,\quad \S^{\sharp}(q)=1.\]
  If $r/s=\infty$, we set
  \[\R^{\flat}(q)=1,\quad \S^{\flat}(q)=1-q, \quad
    \R^{\sharp}(q)=1, \quad \S^{\sharp}(q)=0.\]
\end{mydef}

\begin{remark}
Note that $\R^{\flat}(q)$, $\S^{\flat}(q)$, $\R^{\sharp}(q)$, and $\S^{\sharp}(q)$ depend on \emph{both} $r$ and $s$.
\end{remark}

The continued fraction expansion of every $q$-deformed rational number induces an element of $\PSL_{2,q}(\Z)$ analogous to that in \cref{continuedmatrices}.
Right-multiplying this element by $\infty$ (for right $q$-rationals) or $1/(1-q)$ (for left $q$-rationals) returns the original $q$-rational.
The following proposition expresses this element of $\PSL_{2,q}(\Z)$ in terms of the numerators and denominators of the corresponding $q$-rationals.
The second part of this proposition appears as~\cite[Proposition 4.3]{mor.ovs:20}.
\begin{prop}[Matrix formula for $q$-rationals]\label{qmatrices}
  Consider any rational number $r/s$ with continued fraction expansion $\mathbf{a}=[a_{1},...,a_{2n}]$.
  Let $r'/s'$ be the rational number with continued fraction expansion $[a_{1},...,a_{2n-1}]$. 
  \begin{enumerate}
  \item We have the following for left $q$-rationals.
    \[q^{a_{2}+a_{4}+...+a_{2n-2}}\b_{\mathbf{a},q}
      \begin{bmatrix}1 & 1-q^{-1} \\ 1-q&1  \end{bmatrix}=
      \begin{cases}
        \begin{bmatrix}q^{-a_{2n}}\R^{\flat}(q) & \R^{\flat\prime}(q) \\ q^{-a_{2n}}\S^{\flat}(q) & \S^{\flat\prime}(q) \end{bmatrix}, & \displaystyle 0\leq \frac{r}{s}<\infty,\\[3ex]
        \begin{bmatrix}
          q^{-a_{2n}+1}\R^{\flat}(q) & \R^{\flat\prime}(q) \\ q^{-a_{2n}+1}\S^{\flat}(q) & \S^{\flat\prime}(q)
        \end{bmatrix}, & \displaystyle \frac{r}{s}<0\text{ or } \frac{r}{s}=\infty.
\end{cases}\]
\item We have the following for right $q$-rationals. \[q^{a_{2}+a_{4}+...+a_{2n}}\b_{\mathbf{a},q}=
    \begin{cases}
      \begin{bmatrix}q\R^{\sharp}(q) & \R^{\sharp\prime}(q) \\ q\S^{\sharp}(q) & \S^{\sharp\prime}(q) \end{bmatrix},& \displaystyle 0\leq \frac{r}{s}<\infty,\\[3ex]
      \begin{bmatrix}\R^{\sharp}(q) & \R^{\sharp\prime}(q) \\ \S^{\sharp}(q) & \S^{\sharp\prime}(q) \end{bmatrix}, & \displaystyle\frac{r}{s}<0 \text{ or } \frac{r}{s}=\infty.
    \end{cases}\]
\end{enumerate}
\end{prop}

\subsection{Limits of $q$-deformed rational numbers}\label{subsec:limits-q-rationals}
The following theorem partially formalises the description of $q$-deformed rational numbers depicted in \cref{fig:stdfarey} and \cref{fig:qfarey}. The vertices of the $q$-Farey tessellation are the right $q$-deformed rationals, while the left $q$-deformed rational numbers arise as limits of sequences of right $q$-deformed rationals. This theorem is worth comparing to Morier-Genoud and Ovsienko's study of limiting behaviour of $q$-deformed rational numbers, in terms of Taylor series, in \cite{mor.ovs:19*1}, especially Theorem 1 and Remark 3.2. The proof owes its technique to the proofs in that paper.
\begin{theorem}[Limits of $q$-rationals]\label{qlimits}
Suppose $q<1$. Consider any convergent sequence of rational numbers $(r_{m}/s_{m})_{m \in \N} \in \Q^{\N}$.
\begin{enumerate}
\item If $(r_{m}/s_{m})_{m \in \N}$ converges to an irrational number $t \in \RR \setminus \Q$, then the sequence of right $q$-deformed rational numbers $([r_{m}/s_{m}]_{q}^{\sharp})_{m \in \N} \in \RR^{\N}$ converges to a limit, uniquely determined by $t$.
\item  If $(r_{m}/s_{m})_{m \in \N}$ converges to a rational number $r/s$ from the right, then the sequence of right $q$-deformed rational numbers $([r_{m}/s_{m}]_{q}^{\sharp})_{m \in \N} \in \RR^{\N}$ converges to $[r/s]_{q}^{\sharp}$.
\item If $(r_{m}/s_{m})_{m \in \N}$ converges to a rational number $r/s$ from the left, then the sequence of right $q$-deformed rational numbers $([r_{m}/s_{m}]_{q}^{\sharp})_{m \in \N} \in \RR^{\N}$ converges to $[r/s]_{q}^{\flat}$.
\end{enumerate}
\end{theorem}
By \cref{qlimits} Part (1), for every irrational number $t$, there is a unique $q$-deformation $[t]_{q}$, which is the limit of the $q$-deformation of any rational sequence limiting to $t$. We call the set 
$$\mathbb{I}_{q}:=\{[t]_{q} \mid t \in \RR \sm \Q\}$$ 
the \emph{$q$-deformed irrational numbers}. 
\begin{proof}
  First, we prove (1).

  As noted in \cite[Section 2]{lec.mor.ovs.ea:21}, the negative $q$-rationals are the image of the positive $q$-rationals under the continuous bijection
$$x \mapsto -\frac{q}{x},$$
with $\infty$ sent to $0$ and vice versa. Thus it is sufficient to give a proof for positive $t$.

Let $(r_{n}/s_{n})_{n \in \N}$ converge to an irrational number $t \in \RR \setminus \Q$. It is well-known that every irrational number has a unique infinite continued fraction expansion $[a_{1},a_{2},...]$. We can approximate the sequence $(r_{n}/s_{n})_{n \in \N}$ by the `sequence of convergents' $([a_{1},a_{2},...,a_{n}])_{n \in \N}$. Since $[r/s]_{q}^{\sharp}>[r'/s']_{q}^{\sharp}$ if and only if $r/s>r'/s'$, the right $q$-deformed sequence $([r_{n}/s_{n}]_{q}^{\sharp})_{n \in \N}$ is also approximated by $([a_{1},a_{2},...,a_{n}]_{q}^{\sharp})_{n \in \N}$. Thus, we can assume without loss of generality that
$r_{n}/s_{n}=[a_{1},a_{2},...,a_{n}]$ for all $n \in \N$.
\par
We split $(r_{n}/s_{n})_{n \in \N}$ into two disjoint subsequences. Define the index set $I\subset \N$ by
$$I:=\Big\{n \in \N \mid \frac{r_{n}}{s_{n}}<t\Big\}.$$
Both $\lim_{n \in I}r_{n}/s_{n}$ and $\lim_{n \notin I}r_{n}/s_{n}$ exist, as each sequence is monotone and bounded. Denote the former by $[t]_{q}^{-}$ and the latter by $[t]_{q}^{+}$. 
\par
We would like to show $[t]_{q}^{-}=[t]_{q}^{+}$. It is sufficient to show that 
$$\lim_{n \to \infty}\Big\v\Big[\frac{r_{n}}{s_{n}}\Big]_{q}^{\sharp}-\Big[\frac{r_{n-1}}{s_{n-1}}\Big]_{q}^{\sharp}\Big\v=0.$$ Let $[r_{n}/s_{n}]_{q}^{\sharp}=\R^{\sharp}_{n}(q)/\S^{\sharp}_{n}(q)$, and let $[r_{n-1}/s_{n-1}]_{q}^{\sharp}=\R^{\sharp}_{n-1}(q)/\S^{\sharp}_{n-1}(q)$, using the notation of \cref{qrat}. By the matrix formula for right $q$-deformed rationals (\cref{qmatrices}), we have
$$\Big\v\frac{\R^{\sharp}_{n}(q)}{\S^{\sharp}_{n}(q)}-\frac{\R^{\sharp}_{n-1}(q)}{\S^{\sharp}_{n-1}(q)}\Big\v=q^{a_{2}+a_{4}+...+a_{n}-1}\frac{\v\det \b_{\mathbf{a},q}\v}{\S^{\sharp}_{n}(q)\S^{\sharp}_{n-1}(q)}.$$
By construction, $\S^{\sharp}_{n}(q)$ and $\S^{\sharp}_{n-1}(q)$ both have constant term $1$ and positive coefficients, so $\S^{\sharp}_{n}(q),\S^{\sharp}_{n-1}(q) \geq 1$. And one can compute that 
$$\det \b_{\mathbf{a},q}=q^{a_{1}-a_{2}+...+a_{n-1}-a_{n}},$$
therefore,
$$\Big\v\frac{\R^{\sharp}_{n}(q)}{\S^{\sharp}_{n}(q)}-\frac{\R^{\sharp}_{n-1}(q)}{\S^{\sharp}_{n-1}(q)}\Big\v\leq q^{a_{1}+a_{3}+...+a_{n-1}-1}.$$
(Technically this assumes $n$ is even, however an all but identical argument works for $n$ odd.)
Given the assumption that $t$ (and therefore also $\frac{r_{n}}{s_{n}}$ for sufficiently large $n$) are positive, we have $a_{1},...,a_{n-1}\geq 0$, so since $q<1$ we deduce that
$$\lim_{n \to \infty}\Big\v\Big[\frac{r_{n}}{s_{n}}\Big]_{q}^{\sharp}-\Big[\frac{r_{n-1}}{s_{n-1}}\Big]_{q}^{\sharp}\Big\v=0.$$
To prove that the limit is uniquely determined by $t$, we note again that $[r/s]_{q}^{\sharp}>[r'/s']_{q}^{\sharp}$ if and only if $r/s>r'/s'$. Thus, between any two `$q$-irrationals' $[t]_{q}, [t']_{q}$ for $t \neq t'$, there is a right $q$-rational $[r/s]_{q}^{\sharp}$, so we cannot have $[t]_{q}=[t']_{q}$.
\par
Now we prove (2).  As in the irrational case, we can assume that $r/s$ is positive. Let $[a_{1},...,a_{2n}]$ be the continued fraction expansion of $r/s$. Any right limiting sequence approaching $r/s$ can be approximated by the sequence $(r_{m}/s_{m})_{m \in \N}$, where $r_{m}/s_{m}$ has continued fraction expansion $\mathbf{a}_{m}=[a_{1},...,a_{2n}-1,1,m]$ (unless $a_{2n}=1$, in which case $r_{m}/s_{m}$ has expansion $\mathbf{a}_{m}=[a_{1},...,a_{2n-1}+1,m]$). Since $[r/s]_{q}^{\sharp}>[r'/s']_{q}^{\sharp}$ if and only if $r/s>r'/s'$, the corresponding right $q$-deformed sequence is also approximated by $([r_{m}/s_{m}]_{q}^{\sharp})_{m \in \N}$.
\par
Let
$[{r}/{s}]_{q}^{\sharp}=\R^{\sharp}(q)/\S^{\sharp}(q)$, and
$[r_{m}/s_{m}]_{q}^{\sharp}=\R_{m}^{\sharp}(q)/\S_{m}^{\sharp}(q)$, as defined in \cref{qrat}
We view the right $q$-deformed rational numbers as vectors in projective space. Applying the matrix formula for right $q$-rationals (\cref{qmatrices}), we have
\begin{align*}
\lim_{m \to \infty}\begin{bmatrix} \frac{q^{1-m}}{[m]_{q^{-1}}}\R_{m}^{\sharp}(q) \\ \frac{q^{1-m}}{[m]_{q^{-1}}}\S_{m}^{\sharp}(q) \end{bmatrix}&=\lim_{m \to \infty}\frac{1}{[m]_{q^{-1}}}q^{a_{2}+a_{4}+...+a_{2n}}\b_{\mathbf{a}_{m},q}\begin{bmatrix}1 \\ 0\end{bmatrix}
\\&=q^{a_{2}+a_{4}+...+a_{2n}}\s_{1,q}^{-a_{1}}\s_{2,q}^{a_{2}}...\s_{1,q}^{-a_{2n-1}}\s_{2,q}^{a_{2n}-1}\s_{1,q}^{-1}\begin{bmatrix} 0 & 0\\ 1 & q^{-1} -1\end{bmatrix} \begin{bmatrix} 1 \\ 0 \end{bmatrix}
\\ &=\begin{bmatrix}\R^{\sharp}(q)\\ \S^{\sharp}(q)\end{bmatrix}.
\end{align*}
\par
The final claim (3) can be proved in essentially the same way as (2), replacing the sequence $[a_{1},...,a_{2n}-1,1,m]$ with the sequence $[a_{1},...,a_{2n},m]$.
\end{proof}
\begin{remark}
In the preceding proof, we found that right $q$-deformed rationals correspond to limits of sequences of the form $([a_{1},...,a_{2n+1},m]_{q}^{\sharp})_{m \in \Z}$, and left $q$-deformed rationals correspond to limits of sequences of the form $([a_{1},...,a_{2n},m]_{q}^{\sharp})_{m \in \Z}$. Note that, for ordinary continued fractions,
\begin{align*}
\lim_{m \to \infty}[a_{1},...,a_{2n+1},m]=[a_{1},...,a_{2n+1}], &&\lim_{m \to \infty}[a_{1},...,a_{2n},m]=[a_{1},...,a_{2n}].
\end{align*}
Thus, in a sense, right $q$-rationals correspond to odd continued fraction expansions, and left $q$-rationals to even continued fraction expansions.
\end{remark}
\subsection{Topology of the $q$-Farey tessellation}\label{subsec:topology-q-tessellation}
We briefly summarise here the most important topological feature of the $q$-deformed Farey tessellation: it is an open disk, and the boundary of its closure is a circle, consisting of exactly the $q$-deformed rational curves, and the $q$-deformed irrational numbers. Recall that $\Q_{q}$ denotes the union of the $q$-deformed rational curves (including their endpoints, the left and right $q$-deformed rational numbers), and that $\mathbb{I}_{q}$ denotes the $q$-deformed irrational numbers.
\par
\begin{prop}\label{qtopology}
Let $L_{q}$ denote the proper subset of $\HH$ covered by the $q$-deformed Farey tessellation.
\begin{enumerate}
\item The set $L_{q}$ is homeomorphic to an open disk.
\item The boundary $\dd \ol{L}_{q}$ is equal to $\Q_{q} \sqcup \mathbb{I}_{q}$.
\item There  is a homeomorphism $\Q_{q} \sqcup \mathbb{I}_{q} \to \RR \cup \{\infty\}$, which restricts to the identity on the left and right $q$-deformed rationals, and the $q$-deformed irrationals. So, by Part (2), $\dd \ol{L}_{q} \cong \RR\cup \{\infty\}$.
\item The set $\Q_{q}$ is dense in $\dd\ol{L}_{q}$.
\end{enumerate}
\end{prop}
We expect the following conjecture to hold as well.
\begin{conjecture}
The closure of the $q$-deformed Farey tessellation $\ol{L}_{q}$ is homeomorphic to a closed disk.
\end{conjecture}
Because the $q$-deformed rational curves are dense in the boundary of $\ol{L}_{q}$, these curves are arguably the true $q$-deformed analogues of the rational numbers, hence the notation $\Q_q$ for their union.
\begin{proof}[Proof of \cref{qtopology}]
Part (1) of \cref{qtopology} is true by construction; the $q$-deformed Farey tessellation is homeomorphic to the classical Farey tessellation, which is homeomorphic to an open disk. 
\par
The homological counterpart of Part (2) is \cref{boundary} below, and the proof of Part (2) is virtually the same as the proof of \cref{boundary}. 
\par
We will prove Part (3). Part (4) follows immediately from Part (3) and \cref{qlimits}.
\par
It is not hard to check that for any two real numbers $t<t'$ (which may be rational or irrational), we have the following inequalities:
\begin{equation}\label{order}[t]_{q}^{\flat}\leq[t]_{q}^{\sharp} <[t']_{q}^{\flat}\leq[t']_{q}^{\sharp}. \end{equation}
The left (resp. right) weak inequality is strict if and only if $t$ (resp. $t'$) is rational. Define an embedding $\iota:\Q_{q} \sqcup \mathbb{I}_{q} \to \RR \cup\{\infty\}$ as follows. For $t \in \RR \cup \{\infty\}$,
$$\iota([t]_{q}^{\flat}) = [t]_{q}^{\flat}, \quad \iota([t]_{q}^{\sharp}) = [t]_{q}^{\sharp}.$$
Extend $\iota$ to $\Q_{q}$ by continuously mapping the hyperbolic geodesic between $[r/s]_{q}^{\flat}$ and $[r/s]_{q}^{\sharp}$ onto the real interval between $[r/s]_{q}^{\flat}$ and $[r/s]_{q}^{\sharp}$. By \cref{order}, $\iota$ is indeed an embedding.
\par
To complete the proof, we just need to show that $\iota$ is surjective. Consider any $p \in \RR \cup \{\infty\}$. For convenience, assume that $0 < p <\infty$. Suppose that $p \notin \iota(\Q_{q})$; we will show that $p \in \mathbb{I}_{q}$.
By assumption, $p$ does not lie in the interval in $\RR \cup \{\infty\}$ between $[\infty]_q^{\flat}$ and $[\infty]_q^{\sharp}$.
By the proof of \cref{qlimits}, the sequence $[n]^{\sharp}_{q}$ for $n \geq 0$ is a monotone increasing sequence converging to $[\infty]_{q}^{\flat}$. 
So, there exists $a_{1} \geq 0$ such that 
$$[a_{1}]_{q}^{\sharp} < p < [a_{1}+1]_{q}^{\sharp}=[a_{1},1]^{\sharp}.$$
\par
The sequence $[a_{1},n]_{q}^{\sharp}$ for $n\geq 1$ is a monotone decreasing sequence converging to $[a_{1}]_{q}^{\sharp}$, by the proof of \cref{qlimits}. Since $p >[a_{1}]_{q}^{\sharp}$, there exists $a_{2} \geq 1$ such that
$$[a_{1},a_{2}+1]_{q}^{\sharp}<p<[a_{1},a_{2}]_{q}^{\sharp}.$$
\par
Continuing in this manner, we obtain an infinite $q$-deformed continued fraction expansion $[a_{1},a_{2},...]_{q}^{\sharp}$ converging to $p$.
By the first part of~\cref{qlimits}, the number $p$ is the $q$-deformation of the irrational number with continued fraction expansion $[a_{1},a_{2},...]$, which proves that $p \in \mathbb{I}_{q}$ as desired.
This completes the proof.
\end{proof}
\section{A homological interpretation of $q$-deformed rational numbers}\label{sec:homological}
The aim of this section is to give a concrete homological interpretation of $q$-deformed rational numbers.
Our main tool will be an abstract gadget called a Harder--Narasimhan automaton, whose construction depends on the choice of a Bridgeland stability condition.
We write down a particular Harder--Narasimhan automaton for the 2-Calabi--Yau category associated to the $A_2$ quiver, which yields the desired application.
We recall details about Bridgeland stability conditions and Harder--Narasimhan automata in~\cref{app:stability-and-automata}.
\subsection{The 2-Calabi--Yau category associated to the $A_{2}$ quiver}\label{subsec:c2-automaton}
We follow the construction of the 2-Calabi--Yau category associated to the $A_2$ quiver given in \cite[\textsection 4a]{sei.tho:01} and \cite[\textsection 2.1]{bap.deo.lic:20}.
Let $P(A_2)$ denote the path algebra of the double of the $A_2$ quiver.  This quiver has two vertices, denoted $1$ and $2$, and two oriented edges, one from $1$ to $2$ and the other from $2$ to $1$.  We denote by $Z_{2}$ the quotient of this path algebra by the two-sided ideal generated by all paths of length 3. We regard $Z_2$ as a differential graded algebra with grading by path length, and zero differential. Let $\DGM_{2}$ be the category of differential graded modules over $Z_{2}$, and let $\D_{2}$ be the derived category of differential graded modules over $Z_{2}$ (obtained from $\DGM_{2}$ by inverting quasi-isomorphisms; note that this is not the same as the standard derived category of the abelian category $\DGM_{2}$). For $i=1,2$, let $P_{i}$ denote the differential graded module $Z_{2}(i)$, where $(i)$ is the trivial path at the vertex $i$. The objects 
$\{P_{1},P_{2}\}$ form an $A_{2}$ configuration of spherical objects in $\D_{2}$, in the sense of \cite[Definition 1.1]{sei.tho:01}.
Let $\C_{2}$ denote the full triangulated subcategory of $\D_{2}$ generated by $P_{1},P_{2}$ under extensions. Then $\C_{2}$ is a 2-Calabi--Yau category.
\par
Especially important for our purposes are the following structural features of $\C_{2}$.
\begin{enumerate}
\item
There is a unique morphism (up to scaling) from $P_1$ to $P_2$, and also from $P_2$ to $P_1$.
We denote the cones of these morphisms by $P_{12}$ and $P_{21}$ respectively.
They are indecomposable spherical objects of $\C_{2}$.
\item
The extension closure of $P_{1},P_{2}$ is the heart  of a bounded $t$-structure on $\C_{2}$. We call this $t$-structure the \textit{standard $t$-structure} on $\C_{2}$, and denote its heart by $\hrt_{\std}$.
Objects in the standard heart are direct sums of the indecomposables $P_{1}$, $P_{2}$, $P_{12}$ and $P_{21}$.
\item
Every spherical object $X$ in $\C_{2}$ gives rise to an autoequivalence $\s_{X}$ on $\C_{2}$ known as the spherical twist along $X$. These autoequivalences form a group, generated by $\s_{P_{1}},\s_{P_{2}}$ (henceforth simply $\s_{1},\s_{2}$). The group of spherical twists is isomorphic to the 3-strand braid group $B_{3} = B_{A_2}$, inducing a (weak) braid group action on $\C_{2}$ (see \cite{kho.sei:00} and \cite{sei.tho:01} for details). Let $\bS$ be the set of isomophism classes of spherical objects of $\C_2$.  It is known that $B_{3}$ acts transitively on this set; see e.g. \cite[Theorem 1.1]{bap.deo.lic:21}.  
\end{enumerate}
\subsection{Standard stability conditions on $\C_{2}$}
The proofs in the following section use a choice of a stability condition on $\C_2$.
Any choice will work, but for convenience we choose a degenerate standard stability condition, defined as follows.
\begin{mydef}\label{stddef}
  A stability condition $\t=(\cP, Z)$ is \emph{standard} if $\cP([\f,\f+1))=\hrt_{\std}$ for some $\f \in \RR$.
  A standard stability condition $\t$ is \emph{degenerate} the phase of $P_1$ is equal to the phase of $P_2$.
\end{mydef}
Since $P_1$ and $P_2$ are simple in $\heart_{\std}$, they are automatically stable in any standard stability condition.
In a degenerate standard stability condition, the objects $P_{12}$ and $P_{21}$ are also semistable of the same phase.
\par
We will use the following proposition.
It is a serious result, and corresponding versions of this proposition are true for any 2-Calabi--Yau category of a finite ADE type quiver.
Several proofs of versions of this proposition are known, in various levels of generality: see, e.g., \cite[Theorem 1]{ish.ued.ueh:10}, \cite[Proposition 4.13]{ike:14}, \cite[Corollary 9.5]{ada.miz.yan:19}, \cite[Theorem 1.2]{bap.deo.lic:21}.
\begin{prop}\label{tessellation}\
  \begin{enumerate}
  \item The space $\Stab(\C_2)$ is connected.
  \item For any stability condition $\t \in \Stab(\C_{2})$, there exists a  braid $\b \in B_{3}$ such that $\b \t$ is a standard stability condition.
  \end{enumerate}
\end{prop}
\par
From now on, we fix a degenerate standard stability condition $\tau$.
The spherical objects in $\C_{2}$ admit a simple classification, based on their $\t$-HN filtration factors. This classification will be a key tool in later proofs.
 Given a set $S \subset \Ob \C_2$, let \([s \mid s \in S]\) denote the set of all objects of $\C_2$ whose $\tau$-HN filtration factors consist of shifted copies of elements of $S$. 
 For example, $[P_1,P_{12}]$ consists of all objects whose $\tau$-HN filtration factors are shifts of either $P_1$ or $P_{12}$.

 The following proposition is similar to \cite[Proposition 5.8]{bap.deo.lic:20}.
 \begin{prop}
   For  each spherical object $X \in \bS$, one of the following is true:
\begin{enumerate}
\item $X \in [P_{2}, P_{21}],$
\item $X \in [P_{21}, P_{1}],$
\item $X \in[P_{1}, P_{12}],$
\item $X \in  [P_{12}, P_{2}].$
\end{enumerate}
 \end{prop}
\subsection{An automaton for the category $\C_2$}\label{theautomaton}
In this section we write down an explicit HN automaton for $\C_2$ (see~\cref{def:hn-automaton}).
In the next section we discuss its implications on the dynamics of the braid group action on spherical objects of $\C_2$.
Once again, fix $\t$ to be a degenerate standard stability condition.
To construct a $\tau$-HN automaton for $\C_2$, we specify the following pieces of data.
\begin{enumerate}
\item \cref{fig:b3-automaton} simultaneously depicts a $B_3$-labelled graph $\Theta$ and a $\Theta$-set $\underline{S}\subset \underline{\Ob C_2}$.
\begin{figure}[h]
  \tikzcdset{every label/.append style = {font = \small}}
  \begin{tikzcd}[row sep=4em, column sep=4em]
    {[P_1,P_{12}]} \arrow[yshift=0.3em]{r}{\sigma_2} \arrow[xshift=0.3em]{d}{\sigma_2^{-1}} \arrow[loop left]{}{\sigma_1}
    & {[P_{21},P_1]} \arrow[yshift=-0.3em]{l}{\sigma_2^{-1}} \arrow[xshift=0.3em]{d}{\sigma_2} \arrow[loop right]{}{\sigma_1^{-1}} \\
    {[P_{12},P_2]} \arrow[xshift=-0.3em]{u}{\sigma_1}\arrow[yshift=0.3em]{r}{\sigma_1^{-1}} \arrow[loop left]{}{\sigma_2^{-1}}
    & {[P_2,P_{21}]} \arrow[xshift=-0.3em]{u}{\sigma_1^{-1}}\arrow[yshift=-0.3em]{l}{\sigma_1} \arrow[loop right]{}{\sigma_2}
  \end{tikzcd}
  \caption{A $B_3$-labelled graph $\Theta$, together with an assignment of a subset of $\Ob \C_2$ to each vertex.}\label{fig:b3-automaton}
\end{figure}
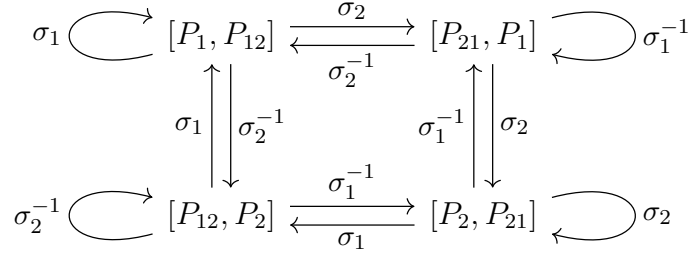
\item \cref{fig:b3-automaton-rep} depicts a $\Theta$-representation $\underline{M}$.
  \begin{figure}[h]
  \tikzcdset{every label/.append style = {font = \small}}
  \begin{tikzcd}[row sep=8em, column sep=8em, ampersand replacement=\&]
    {\mathbb{Z}[q^\pm]^2}
    \arrow[yshift=0.3em]{r}{\id}
    \arrow[xshift=0.3em]{d}{\begin{pmatrix}q^{-1}& 0\\1&1\end{pmatrix}}
    \arrow[loop left]{}{\begin{pmatrix}q^{-1}& q^{-1}\\0&1\end{pmatrix}}
    \& {\mathbb{Z}[q^\pm]^2}
    \arrow[yshift=-0.3em]{l}{\id}
    \arrow[xshift=0.3em]{d}{\begin{pmatrix}1& 1\\0&q\end{pmatrix}}
    \arrow[loop right]{}{\begin{pmatrix}1& 0\\q&q\end{pmatrix}}
    \\
    {\mathbb{Z}[q^\pm]^2}
    \arrow[xshift=-0.3em]{u}{\begin{pmatrix}1& 1\\0&q\end{pmatrix}}
    \arrow[yshift=0.3em]{r}{\id}
    \arrow[loop left]{}{\begin{pmatrix}1& 0\\q&q\end{pmatrix}}
    \& {\mathbb{Z}[q^\pm]^2}
    \arrow[xshift=-0.3em]{u}{\begin{pmatrix}q^{-1}& 0\\1&1\end{pmatrix}}
    \arrow[yshift=-0.3em]{l}{\id}
    \arrow[loop right]{}{\begin{pmatrix}q^{-1}& q^{-1}\\0&1\end{pmatrix}}
  \end{tikzcd}
  \caption{A $\Theta$-representation $\underline{M}$.}\label{fig:b3-automaton-rep}
\end{figure}
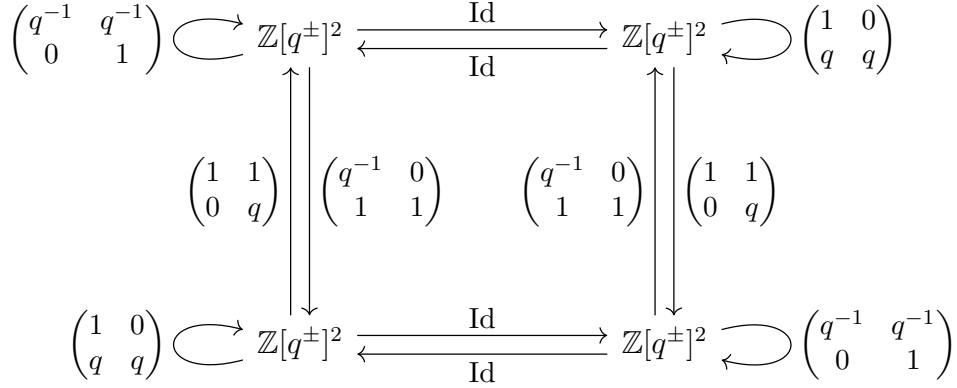
As shown in~\cref{fig:b3-automaton}, the set at every vertex has the form $[A,B]$, where $A,B\in \Sigma_\tau$.
The summands in the module $\mathbb{Z}[q^\pm]^2 = \mathbb{Z}[q^\pm] \oplus \mathbb{Z}[q^\pm]$ at that vertex should be thought of as being indexed by $A$ and $B$ respectively.

\item The morphism $i \colon \underline{S} \to \underline{M}$ is as follows.
At each vertex, the map $[A,B] \to \mathbb{Z}[q^\pm]^2$ takes an object $X \in [A,B]$ to the $A$ and $B$ coordinates of $\HN_\tau(X)$.
The proof that $i$ is $\Theta$-equivariant is a calculation that proceeds exactly as in the proof of~\cite[Proposition 5.1]{bap.deo.lic:20}.
\item Finally, $h_v \colon M_v \to \mathbb{Z}[q^{\pm}]^{\Sigma_\tau}$ is simply the inclusion map.
We need to check that for any vertex $v$ and any $X \in S_v$, we have
\[\HN_\tau(X) = h_v(i_v(X)).\]
In other words, that the $\tau$-HN multiplicity vectors transform according to matrices in~\cref{fig:b3-automaton-rep}.
This check, which is involved, is essentially identical to the one done in~\cite[Proposition 5.1]{bap.deo.lic:20}.
We omit the details here.
\end{enumerate}
\subsection{The dynamics of the braid group action on spherical objects}\label{subsec:dynamics}
We now define two kinds of functionals $\bS\times\bS \to \Z[q^{\pm}]$.  
The first of these is an extension of a functional $\{P_{1},P_{2}\}\times \bS \to \Z[q^{\pm}]$ that counts the `occurences' of the building block complexes $P_{2}$ and $P_{1}$, respectively, in a complex $X \in \bS$.
We continue to denote by $\t$ a degenerate standard stability condition on $\C_{2}$.

We know that in this case, the set of indecomposable semistable objects in the heart is just
\[\Sigma_\tau = \{P_1, P_2, P_{12}, P_{21}\}.\]
With these coordinates, for any $X \in \T$, we write $\HN_\tau(X)$ as
\[\HN_\tau(X) = (\pi_1(X), \pi_2(X), \pi_{12}(X), \pi_{21}(X)).\]
\begin{mydef}
We define functionals $\bS \times \bS \to \Z[q^{\pm}]$ as follows. 
For any $X \in \bS$,
\begin{align*}
\occ_{q}(P_{1}, X)&:=\p_{2}(X)+\p_{12}(X)+\p_{21}(X),\\
\occ_{q}(P_{2}, X)&:=\p_{1}(X)+\p_{12}(X)+\p_{21}(X).
\end{align*}
For any $X,Y \in \bS$, let $X=\b P_{1}$, then
$$\occ_{q}(X,Y)=\occ_{q}(P_{1},\b^{-1}Y).$$
\end{mydef}
\begin{remark}
 Note that $\occ_{q}(P_{1},X)$ counts the `occurrences' of $P_{2}$ in $X$ and vice versa, which is notationally confusing but somewhat unavoidable. Every occurence of $P_{2}$ in $X$ must correspond to an HN filtration factor of the form $P_{2}$, $P_{12}$ or $P_{21}$, up to shift. Thus, summing the number of these factors, with degree, gives the number of $P_{2}$'s, with degree.
\end{remark}
\begin{mydef}The functional $\occ_{q}$ has a close cousin, $\ol{\hom}_{q}$:
$$\ol{\hom}_{q}(X,Y):=\begin{cases}
q^{n}(q^{-2}-q^{-1}), & Y\cong X[n],\\
\sum_{n \in \Z}\dim_{\k}\Hom(X,Y[n])q^{-n}, & \text{otherwise.}
\end{cases}$$
\end{mydef}
The following theorems relate these functionals to the $q$-deformed rational numbers. For reasons that will soon be obvious, for any $X \in \bS$ we write $X\geq 0$ to mean $X \in [P_{2}, P_{21}]\cup[P_{21}, P_{1}],$ and $X \leq0$ to mean $X \in [P_{1}, P_{12}]\cup [P_{12}, P_{2}].$

\begin{theorem}\label{qrz1}
For every $X \in \bS$,
$$(-q)^{-\e}\frac{\occ_{q}(P_{2},X)}{\occ_{q}(P_{1},X)},$$
is a right $q$-deformed rational number, 
where
$$\e=\begin{cases} 0, & X\geq 0,\\
1, & X\leq0.\end{cases}$$
Moreover, the induced map from $\bS$ to the set of $q$-deformed rationals is $B_{3}$-equivariant, and bijective (up to shift).
\end{theorem}

\begin{theorem}\label{qrz2}
For every $X \in \bS$,
$$(-1)^{\e}q^{\e-1}\frac{\ol{\hom}_{q}(X,P_{2})}{\ol{\hom}_{q}(X,P_{1})},$$
is a left $q$-deformed rational number, where
$$\e=\begin{cases} 0, & X\geq0, X \not\cong P_{2}[n],\\
1, & X\leq 0, X \not\cong P_{1}[n].\end{cases}$$
Moreover, the induced map from $\bS$ to the set of left $q$-deformed rationals is $B_{3}$-equivariant, and bijective (up to shift).
\end{theorem}
These theorems were first proven in the case $q=1$, where they are identical, by Rouquier and Zimmermann \cite[Proposition 4.8]{rou.zim:03}.
The proofs of these theorems for arbitrary $q$ will occupy the rest of this section.  In \cite[Proposition 5.1]{bap.deo.lic:20}, a proof of the original Rouquier--Zimmermann bijection is established using the theory of Harder--Narasimhan automata; it is this proof technique that we adapt to prove the above theorems. 
More precisely, to prove \cref{qrz1}, we will use the automaton introduced in \cref{theautomaton} to analyse the dynamics of spherical objects in $\C_{2}$ under the braid group action.
\par
The proof of \cref{qrz1} will proceed by an induction on the length of elements of $B_{3}$. To facilitate this induction, we introduce a normal form for the three-strand braid group $B_{3}$, which is (mostly) recognised by the automaton (that is, any braid in this form that we actually need corresponds to a path along the edges of the automaton). 
Let $\w:=\s_{2}\s_{1}\s_{2}\s_{1}\s_{2}\s_{1}$ be the generator of the centre of $B_{3}$, 
\begin{lemma}[Continued normal form]\label{continued}
Consider any braid $\b \in B_{3}$. We can write $\b$ uniquely in one of four forms, determined by the image of $P_{1}$:
\begin{numcases}{\b=}
\s_{1}^{-a_{1}}\s_{2}^{a_{2}}...\s_{1}^{-a_{2n-1}}\s_{2}^{a_{2n}}\s_{1}^{M}\o^{N}, & $\b P_{1}\geq 0, \b P_{1} \neq P_{1}[m], P_{2}[m] ,$ \label{continued1} \\
\s_{1}^{a_{1}}\s_{2}^{-a_{2}}...\s_{1}^{a_{2n-1}}\s_{2}^{-a_{2n}}\s_{1}^{M}\o^{N}, & $\b P_{1} \leq 0, \b P_{1} \neq P_{1}[m], P_{2}[m],$ \label{continued2}\\
\s_{1}^{M}\o^{N}, & $\b P_{1} =P_{1}[m],$ \label{continued3} \\
\s_{1}\s_{2}\s_{1}^{M}\o^{N},& $\b P_{1}=P_{2}[m],$\label{continued4}
\end{numcases}
where $a_{1} \in \N$, $a_{2},...,a_{2n} \in \N\sm0$, $M,N \in \Z$.
\end{lemma}
Taking $P_{1}$ as the counterpart of $\infty$, and $P_{2}$ as the counterpart of 0, \cref{continued} corresponds exactly to the continued fraction expansion in \cref{subsec:qrationals-def}.
\begin{proof}
Consider any braid $\b \in B_{3}$. Let $\ol{\b}$ denote the corresponding element of $\PSL_{2}(\Z)$ (under the standard surjection $B_{3} \to \PSL_{2}(\Z)$.) Let $\ol{\b}\big[\begin{smallmatrix}1 \\ 0 \end{smallmatrix}\big]=\big[\begin{smallmatrix} r \\ s \end{smallmatrix}\big]$. Then by \cref{continuedmatrices}, there exists a matrix $\wh{\b'}$ of the form 
$$\ol{\b'}=\ol{\s}_{1}^{-a_{1}}\ol{\s}_{2}^{a_{2}}...\ol{\s}_{1}^{-a_{2n-1}}\ol{\s}_{2}^{a_{2n}},$$
such that $\ol{\b'}\big[\begin{smallmatrix} 1 \\ 0 \end{smallmatrix}\big]=\big[\begin{smallmatrix} r \\ s \end{smallmatrix}\big]$. 
Since $\ol{\b'}$ and $\ol{\b}$ act identically on $\big[\begin{smallmatrix} 1 \\ 0 \end{smallmatrix}\big]$, we have that $\ol{\b'}=\ol{\b}\ol{\s}_{1}^{M}$ for some $M \in \Z$ (this can be deduced from the fact that every matrix in $\PSL_{2}(\Z)$ has determinant 1).
\cref{continued} then follows from the fact that the kernel of $B_{3} \to \PSL_{2}(\Z)$ is the centre $Z(B_{3})$.
\end{proof}
If $\b$ is given by the form in \cref{continued}, we say that $\b$ is written in (even) continued form.
We say that $\b$ is written in strict continued form if any of the following hold:
\begin{enumerate}
\item $\b$ is written in form \eqref{continued1}, $M \leq 0$;
\item $\b$ is written in form \eqref{continued2}, $M \geq 0$;
\item $\b$ is written in form \eqref{continued3};
\item $\b$ is written in form \eqref{continued4}, $M\leq0$.
\end{enumerate}
The following corollary of \cref{continued} follows from the calculation:
$$
\s_{1}^{M}\o^{N}P_{1}=P_{1}[-2N-M].$$
\begin{cor}[Strict continued normal form]\label{continuedcor}
For any $X \in \bS$, there exists a braid $\b$ in strict continued form such that $\b P_{1}=X$.
\end{cor}
\cref{qrz1} is an immediate corollary of the matrix formula for right $q$-rationals (\cref{qmatrices}), and the following \cref{rzprop,rzmess}:
\begin{prop}\label{rzprop}
Consider any $\b \in B_{3}$ which can be written in strict continued form, and let $\ol{\b}$ be the corresponding $q$-deformed matrix. Up to multiplication by $-1$ and by powers of $q$, we have an equality:
$$
\ol{\b}=
\begin{bmatrix}
\occ_{q}(P_{2}, \b P_{1}) & (-q)^{-\e}\occ_{q}(P_{2}, \b P_{2}) \\
(-q)^{\e}\occ_{q}(P_{1}, \b P_{1}) & \occ_{q}(P_{1}, \b P_{2})
\end{bmatrix},
$$
where
$$\e=\begin{cases} 0, & \b P_{2}\geq 0, 
\\ 1, & \b P_{2}\leq 0. \end{cases}$$
\end{prop}
\begin{prop}\label{rzmess}
Consider any $\b \in B_{3}$, which cannot be written in strict contined form, then, up to multiplication by $-1$ and by powers of $q$, we have:
$$\ol{\b}=\begin{bmatrix} 
\occ_{q}(P_{2}, \b P_{1}) & (-q)^{\e-1}\occ_{q}(P_{2}, \b P_{2}) \\
\occ_{q}(P_{1},\b P_{1} ) & (-q)^{2 \e -1}\occ_{q}(P_{1}, \b P_{2})
\end{bmatrix}
$$
where
$$\e=\begin{cases} 0, & \b P_{2}\geq 0, 
\\ 1, & \b P_{2} \leq 0. \end{cases}$$
\end{prop}
\begin{proof}[Proof of \cref{rzprop}]
Consider any $\b \in B_{3}$ which can be written in strict continued form. 
We induct on the length of $\b$.
The base case $\b=\Id$ is trivial.
Since multiplication by $\w$ corresponds to scalar multiplication by a power of $q$, we can assume without loss of generality that the $\w^{N}$-term in $\b$ is trivial (i.e. $N=0$).
\par
For the induction argument, note that $\b$, read from right to left, defines a path in $\Theta$ (\cref{fig:b3-automaton}).
We will ignore the case that $\b$ is in the forms \eqref{continued3} and \eqref{continued4} (in the notation of \cref{continued}), since this can be easily done by hand.
\par
Let $\b_{l}$ be the braid consisting of the last $l$ elementary braids in $\b$, and suppose that
$$\ol{\b_{l}}=\begin{bmatrix}
\occ_{q}(P_{2}, \b_{l} P_{1}) & (-q^{-1})^{\e}\occ_{q}(P_{2}, \b_{l} P_{2}) \\ (-q)^{\e}\occ_{q}(P_{1}, \b_{l} P_{1}) & \occ_{q}(P_{1}, \b_{l} P_{2}).
\end{bmatrix}$$
There are eight cases, depending on whether $\b$ has the form \eqref{continued1} or \eqref{continued2} in the notation of \cref{continued}, and on the value of $\b_{l+1}\b_{l}^{-1}$. Note that $\b_{l}P_{1} \in [P_{i}, P_{j}]$ if and only if  $\b_{l}P_{2} \in [P_{i}, P_{j}]$. 
For illustrative purposes, we will assume that $\b$ has the form \eqref{continued1} (and thus $\e=0$), that $\b_{l}P_{1}, \b_{l}P_{2}\in [P_{21}, P_{1}]$, and $\b_{l+1}=\s_{2}\b_{l}$.
First we calculate
\begin{align*}
\ol{\b_{l+1}}&=\ol{\s_{2}\b_{l}}
\\ &=\begin{bmatrix} 1 & 0 \\ 1 & q^{-1}\end{bmatrix}\begin{bmatrix}
\occ_{q}(P_{2}, \b_{l} P_{1}) & \occ_{q}(P_{2}, \b_{l} P_{2}) \\\occ_{q}(P_{1}, \b_{l} P_{1}) & \occ_{q}(P_{1}, \b_{l} P_{2}).
\end{bmatrix}
\\&=\begin{bmatrix}
\occ_{q}(P_{2}, \b_{l} P_{1}) & \occ_{q}(P_{2}, \b_{l} P_{2}) \\ \occ_{q}(P_{2}, \b_{l} P_{1})+q^{-1}\occ_{q}(P_{1}, \b_{l} P_{1}) & \occ_{q}(P_{2}, \b_{l} P_{2})+q^{-1}\occ_{q}(P_{1}, \b_{l} P_{2}).
\end{bmatrix}
\end{align*}

On the other hand, by definition we have equalities

\begin{align*}
\begin{bmatrix} \p_{21}(\b_{l}P_{1}) &  \p_{21}(\b_{l}P_{2}) \\
 \p_{1}(\b_{l}P_{1}) & \p_{1}(\b_{l}P_{2})
\end{bmatrix}=\begin{bmatrix}
\occ_{q}(P_{1}, \b_{l}P_{1}) & \occ_{q}(P_{1}, \b_{l}P_{2})
\\ \occ_{q}(P_{2}, \b_{l}P_{1})-\occ_{q}(P_{1}, \b_{l}P_{1}) & \occ_{q}(P_{2}, \b_{l}P_{2})-\occ_{q}(P_{1}, \b_{l}P_{2})
\end{bmatrix}
\end{align*}
and 
\begin{align}\label{eqa1}
&\begin{bmatrix} \p_{2}(\b_{l+1}P_{1}) & \p_{2}(\b_{l+1}P_{2}) \\
 \p_{21}(\b_{l+1}P_{1})&  \p_{21}(\b_{l+1}P_{2})
\end{bmatrix}
\\ &=\begin{bmatrix}
\occ_{q}(P_{1}, \b_{l+1}P_{1})-\occ_{q}(P_{2}, \b_{l+1}P_{1}) & \occ_{q}(P_{1}, \b_{l+1}P_{2})-\occ_{q}(P_{2}, \b_{l+1}P_{2})
\\ \occ_{q}(P_{2}, \b_{l+1}P_{1}) & \occ_{q}(P_{2}, \b_{l+1}P_{2})
\end{bmatrix}
\end{align}
Applying the linear representation embedded in $\Theta$,
\begin{align}\label{eqa2}
\begin{bmatrix} \p_{2}(\b_{l+1}P_{1}) &  \p_{2}(\b_{l+1}P_{2}) \\
 \p_{21}(\b_{l+1}P_{1}) &  \p_{21}(\b_{l+1}P_{2})
\end{bmatrix}
&=\begin{bmatrix} q^{-1} & 0 \\ 1 & 1 \end{bmatrix}\begin{bmatrix} \p_{21}(\b_{l}P_{1}) & \p_{21}(\b_{l}P_{2}) \\
 \p_{1}(\b_{l}P_{1}) &  \p_{1}(\b_{l}P_{2})
\end{bmatrix}
\\&=\begin{bmatrix} q^{-1}\occ_{q}(P_{1}, \b_{l} P_{1}) & q^{-1}\occ_{q}(P_{1}, \b_{l} P_{2})
\\  \occ_{q}(P_{2}, \b_{l} P_{1}) &  \occ_{q}(P_{2}, \b_{l} P_{2})
\end{bmatrix}
\end{align}
Comparing \cref{eqa1} and \cref{eqa2}, gives us that for $i \in \{1,2\}$,
\begin{align*}
\occ_{q}(P_{2}, \b_{l+1}P_{i})&=\occ_{q}(P_{2}, \b_{l}P_{i}),\\
\occ_{q}(P_{1}, \b_{l+1}P_{i})&=q^{-1}\occ_{q}(P_{1}, \b_{l}P_{i}) +\occ_{q}(P_{2}, \b_{l}P_{i}),
\end{align*}
which proves that
$$\ol{\b_{l+1}}=\begin{bmatrix}
\occ_{q}(P_{2}, \b_{l+1} P_{1}) & \occ_{q}(P_{2}, \b_{l+1} P_{2}) \\ \occ_{q}(P_{1}, \b_{l+1} P_{1}) & \occ_{q}(P_{1}, \b_{l+1} P_{2}).
\end{bmatrix}.$$
Verifying the other seven cases (using the exact same procedure) completes the induction argument.
\end{proof}
The proof of \cref{rzmess} is virtually the same as the proof of \cref{rzprop}. Any braid which cannot be written in strictly normal form is a right multiple of a braid which can be written in strictly normal form by $\s_{1}^{M}$, thus one can prove this case using the same method as in part 1, but changing the base case appropriately.
\par

This completes the proof of \cref{qrz1}.

A proof of \cref{qrz2} can be constructed independently, analogously to the above proof.  In lieu of doing that, 
we will bootstrap \cref{qrz2} to \cref{qrz1}, by way of the following relationship between $\occ_{q}$ and $\ol{\hom}_{q}$:
\begin{lemma}[$\hom$-values]\label{homvalues}
The functional $\ol{\hom}_{q}(P_{1},-)$ takes the following values:
$$\ol{\hom}_{q}(P_{1},X)=\begin{cases} 
(1-q^{-1})\occ_{q}(P_{2},X)+q^{-1}\occ_{q}(P_{1},X),
& X \geq 0\\
(q^{-2}-q^{-1})\occ_{q}(P_{2},X)+q^{-1}\occ_{q}(P_{1},X),
& X \leq 0.
\end{cases}$$
Symmetrically,
the functional $\ol{\hom}_{q}(P_{2}, -)$ takes the following values:
$$\ol{\hom}_{q}(P_{2},X)=\begin{cases} 
(q^{-2}-q^{-1})\occ_{q}(P_{1},X)+q^{-1}\occ_{q}(P_{2},X),
& X\geq 0, \\
(1-q^{-1})\occ_{q}(P_{1},X)+q^{-1}\occ_{q}(P_{2},X),
& X \leq 0.
\end{cases}$$
\end{lemma}
\begin{proof}[Proof of \cref{qrz2}]
Consider any $X \in \bS$. Assume that $X \neq  P_{1},P_{2}$. 
By \cref{continuedcor}, we can write $X=\b P_{1}$ for some $\b$ in strict continued form. 
It is easy to check that $\b P_{1}\geq 0$ implies $\b^{-1} P_{1} \leq 0$ (and vice versa), and thus  by \cref{rzprop} we have the following equality (up to scaling):
\begin{align*}\ol{\b}&=\begin{bmatrix}\occ_{q}(P_{2}, \b^{-1} P_{1}) & (-q)^{\e-1}\occ_{q}(P_{2}, \b^{-1} P_{2}) \\(-q)^{1-\e}\occ_{q}(P_{1}, \b^{-1} P_{1}) & \occ_{q}(P_{1}, \b^{-1} P_{2})\end{bmatrix}^{-1}\\
&=\begin{bmatrix}\occ_{q}(P_{1}, \b^{-1} P_{2}) & (-1)^{\e}q^{\e-1}\occ_{q}(P_{2}, \b^{-1} P_{2}) \\(-1)^{\e}q^{1-\e}\occ_{q}(P_{1}, \b^{-1} P_{1}) & \occ_{q}(P_{2}, \b^{-1} P_{1})\end{bmatrix}.
\end{align*}
Let 
$\R^{\flat}(q)/\S^{\flat}(q)$ be the rational function representation of $[\b(\infty)]_{q}^{\flat}$, using the notation in \cref{qrat}.
Combining the matrix formula for right $q$-rationals (\cref{qmatrices}) and the $\hom$-values lemma (\cref{homvalues}), we have (again up to scaling):
\begin{align*}
\begin{bmatrix}
\R^{\flat}(q) \\ \S^{\flat}(q)
\end{bmatrix} &=
\ol{\b}\begin{bmatrix}q^{-1}\\ q^{-1}-1 \end{bmatrix}
\\ &=\begin{bmatrix} 
q^{-1}\occ_{q}(P_{1}, \b^{-1} P_{2}) +(-1)^{\e}(q^{\e-2}-q^{\e-1})\occ_{q}(P_{2}, \b^{-1} P_{2}) \\
(-1)^{\e}q^{-\e}\occ_{q}(P_{1}, \b^{-1} P_{1}) +(q^{-1}-1)\occ_s{q}(P_{2}, \b^{-1} P_{1})
\end{bmatrix}
\\&=\begin{bmatrix}
\ol{\hom}_{q}(P_{1}, \b^{-1}P_{2})
\\ (-1)^{\e}q^{1-\e}\ol{\hom}_{q}(P_{1}, \b^{-1}P_{1})
\end{bmatrix}
\\ &=\begin{bmatrix}
\ol{\hom}_{q}(X, P_{2})
\\  (-1)^{\e}q^{1-\e}\ol{\hom}_{q}(X,P_{1})
\end{bmatrix}
\end{align*}
This shows that
$$ (-1)^{\e}q^{\e-1}\frac{\ol{\hom}_{q}(X,P_{2})}{\ol{\hom}_{q}(X,P_{1})}$$
is indeed a left $q$-deformed rational number, and that the map sending $\b P_{1}$ to $[\b(\infty)]_{q}^{\flat}$ is a bijection.
\par
We now prove $B_{3}$-equivariance.
By construction, we have that for every $X \in \bS$, there is some $\b \in B_{3}$, such that $X=\b P_{1}$ and
$$ (-1)^{\e}q^{\e-1}\frac{\ol{\hom}_{q}(\b P_{1},P_{2})}{\ol{\hom}_{q}(\b P_{1},P_{1})}=\b\frac{\ol{\hom}_{q}(P_{1},P_{2})}{\ol{\hom}_{q}(P_{1},P_{1})}=\b [\infty]_{q}^{\flat}.$$
Consider any $\b_{1}, \b_{2} \in B_{3}$ such that $\b_{1} P_{1}=\b_{2} P_{1}$. To prove $B_{3}$-equivariance, it will be sufficient to show that $\b_{1}$ and $\b_{2}$ act identically on $[\infty]_{q}^{\flat}$. By elementary braid group arithmetic, $\b_{1}=\b_{2}\s_{1}^{n}$ for some $n \in \Z$. Since $\s_{1}$ fixes $[\infty]_{q}^{\flat}$,
we have that
$$\b_{1}[\infty]_{q}^{\flat}=\b_{2}[\infty]_{q}^{\flat}.$$
This completes the proof of $B_{3}$-equivariance.
\end{proof}

\section{A family of compactifications of the space of Bridgeland stability conditions}\label{sec:compactification}
The aim of this section is to propose a family of compactifications of the space of Bridgeland stability conditions of a triangulated category indexed by $q > 0$.
This construction generalises results of~\cite{bap.deo.lic:20}.  In particular, for each real number $q>0$, the corresponding compactification will be the closure of an embedding (via $q$-masses) of $\Stab(\mathcal T)/\mathbb{C}$ into a fixed infinite projective space.
We prove and conjecture several properties of this construction for 2-CY categories associated to connected quivers, and analyse the construction in detail for the category $\C_2$.

\subsection{The deformed mass embedding}

Suppose that $\mathcal T$ is the 2-CY category associated to a connected quiver (we will specialise to $\C_{2}$ again shortly).
Let $\bS$ be the set of spherical objects of $\mathcal T$.
Let $\RR^\bS$ be the space of real-valued functions on $\bS$.
Let $\P(\RR^\bS)$ denote the projectivisation of this space; that is, $\P(\RR^{\bS}) = (\RR^{\bS} \sm 0)/\RR^\times$.
Recall the definition of the $q$-mass of an object $X$, given a stability condition $\t$, denoted $m_{q,\t}(X)$ (\cref{massdef}).
For any $q > 0$, we have a map
\[ m_q \colon \Stab(\mathcal T)/\mathbb{C} \to \P(\RR^\bS),\]
defined as
\[m_q(\tau) = [m_{q,\tau}(X)]_{X\in \bS}.\]

The following proposition states that a stability condition can be recovered uniquely up to $\mathbb{C}$-action from its $q$-mass vector.
\begin{prop}[Injectivity]
  Let $\mathcal T$ be the 2-CY category associated to a connected quiver.
  Let $\bS$ be the set of spherical objects of $\mathcal T$.
  Then the map $m_q \colon \Stab(\mathcal T)/\mathbb{C} \to \P(\RR^\bS)$ is injective.
\end{prop}
Injectivity in the case that $q = 1$ is proven in~\cite[Proposition 4.1]{bap.deo.lic:20}.
This proof works in general, given $q$-deformed versions of two trigonometric relations: Ikeda's `$q$-deformed triangle inequality' \cite[Lemma 3.6]{ike:20}, and a `$q$-deformed SSS-triangle congruence theorem', which we describe presently.

For any point $z \in \HH\cup \RR^{>0}$, let $\f(z)$ denote the phase of $z$.
For any pair of points $z_{1}, z_{2} \in \HH\cup \RR^{>0}$, say that the triangle associated to $(z_{1},z_{2})$ is the Euclidean planar triangle with vertices $(0, z_{1}, z_{1}+z_{2})$. 
\par Classically, the lengths of the sides of the trangle associated to $(z_{1},z_{2})$ would be $\v z_{1}\v$, $\v z_{2} \v$ and $\v z_{1}+z_{2}\v$. For our purposes however, it is natural to `$q$-deform' this measure. The $q$-deformed length of a side of this triangle will involve both the classical length, and the phase of the corresponding vector. Precisely, for $v\in \{z_{1},z_{2},z_{1}+z_{2}\}$, the $q$-deformed length of the side with classical length $\v v\v$ is $q^{\f(v)}\v v\v$. Ikeda's $q$-deformed triangle inequality (\cite[Lemma 3.6]{ike:20}) says that the $q$-deformed side-length satisfies the triangle inequality: 
$$q^{\f(z_{1}+z_{2})}\v z_{1}+z_{2}\v \leq q^{\f(z_{1})}\v z_{1}\v+q^{\f(z_{2})}\v z_{2}\v.$$
The next lemma says that the classical SSS-property is also preserved under $q$-deformation: the $q$-deformed side-lengths uniquely determine the triangle (up to congruence).
\begin{lemma}[$q$-SSS]\label{trig2}
Let $z_{1}, z_{2}, z_{1}', z_{2}' \in \HH \cup \RR^{>0}$, and suppose that
\begin{align*}
q^{\f(z_{1})}\v z_{1} \v&=q^{\f(z_{1}')}\v z_{1}' \v,\\
q^{\f(z_{2})}\v z_{2} \v&=q^{\f(z_{2}')}\v z_{2}' \v,\\
q^{\f(z_{1}+z_{2})}\v z_{1} + z_{2} \v&=q^{\f(z_{1}'+z_{2}')}\v z_{1}' + z_{2}' \v.
\end{align*}
Then the triangle associated to $(z_{1}, z_{2})$ is congruent to the triangle associated to $(z_{1}', z_{2}')$.
\end{lemma}
\begin{proof}
Assume that $z_{1}=z_{1}'=1$ (by rotating and scaling appropriately).
For convenience, let
\begin{align*}
\f_{2}:=\f(z_{2}), &&
 m_{2}:=\v z_{2} \v, &&
 z_{12}=z_{1}+z_{2}, && \f_{12}=\f(z_{12}), && m_{12}=\v z_{1}+z_{2}\v,&& c=q^{\f_{2}}m_{2}.
\end{align*}
To prove the proposition, it is enough to show that given the assumption, the vector $z_{2}$ 
is uniquely determined by the values $q^{\f_{2}}m_{2}$ and $q^{\f_{12}}m_{12}$.
For any $t \in [0,1]$, let $w_{t}$ be the vector with phase $t$ and length $q^{-t}c$. We explicitly construct a function $T_{q,c}:[0,1] \to \RR$, which maps $t \in [0,1]$ to $q^{\f(1+w_{t})}\v 1+w_{t} \v$.
If $T_{q,c}$ is injective, then since $T_{q,c}(\f_{2})=q^{\f_{12}}m_{12}$, we can derive $\f_{2}$ from $q^{\f_{12}}m_{12}$. Thus given $q^{\f_{12}}m_{12}$ and $q^{\f_{2}}m_{2}$, we can derive $\f_{2}$ and $m_{2}$, which determine $z_{2}$.
\par
By the law of cosines,
$$\f(1+w_{t})
=\frac{1}{\pi}\cos^{-1}\Big(\frac{1+c q^{-t} \cos (\pi t)}{\sqrt{1+(c q^{-t})^{2}+2 cq^{-t} \cos(\pi t)}}\Big),$$
and
$$\v 1+w_{t} \v
=\sqrt{1+(cq^{-t})^{2}+2cq^{-t} \cos(\pi t)}.$$
Thus,
$$T_{q,c}(t)
=q^{\frac{1}{\pi}\cos^{-1}\big(\frac{1+cq^{-t} \cos (\pi t)}{\sqrt{1+(cq^{-t})^{2}+2cq^{-t} \cos(\pi t)}}\big)}\sqrt{1+(cq^{-t})^{2}+2cq^{-t} \cos(\pi t)}.$$
The derivative of $T_{q,c}$ on the interval $[0,1]$ simplifies to
$$T_{q,c}'(t)=-\frac{q^{f(t)}}{\p \sqrt{c^{2}+q^{2t}+2c q^{t} \cos(\p t)}}(cq^{t} \ln^{2}(q)\sin(\p t)+ \p^{2}c q^{t} \sin(\p t)),$$
where
$$f(t)=-t+\frac{\cos^{-1}\Big(\frac{1+c q^{t}\cos(\p t)}{\sqrt{1+(cq^{-t})^{2}+2cq^{-t} \cos(\p t)}}\Big)}{\p}.$$
Since $T'_{q,c}$ is evidently negative, $T_{q,c}$ is injective on $[0,1]$, completing the proof.
\end{proof}
Further, we conjecture that the mass map is a homeomorphism onto its image.
\begin{conjecture}
  Let $\mathcal T$ be the 2-CY category associated to a connected quiver and let $\bS$ be the set of spherical objects of $\mathcal T$. Let $m_q$ denote the mass map for a $q\in (0,\infty)$.  Then
  \begin{enumerate}
  \item The map $m_q$ is a homeomorphism onto its image.
  \item The closure of the image of $m_q$ is a compactification of $\Stab(\mathcal{T})/\CC$ and is a real manifold with boundary.
  \item When $\mathcal{T}$ is a 2-CY category of type ADE, the closure of the image of $m_q$ is homeomorphic to a closed Euclidean ball.
  \end{enumerate}
\end{conjecture}

In the remainder of this section we prove a weaker implication of the above conjecture in type $A_{2}$; we show that the map $m_q$ is a homeomorphism onto its image, that the image of $m_q$ is homeomorphic to an \emph{open} Euclidean disk, and that the boundary of the closure is homeomorphic to a circle.
\par
The analogue of this conjecture for the 2-CY category of the $A_{2}$-quiver at $q=1$ is proved in \cite[Proposition 5.6]{bap.deo.lic:20}.
The proof of the (following) $q$-deformed proposition is essentially the same (again switching the triangle inequality for the $q$-deformed analogue from \cite[Lemma 3.6]{ike:20}).
\begin{prop}[Homeomorphism]
The map $m_{q}\colon  \Stab(\C_{2})/\CC \to \P^{\bS}(\RR)$ is a homeomorphism onto its image. Moreover,  the closure of $m_{q}(\Stab(\C_{2})/\CC)$ is compact in $\P^{\bS}(\RR)$.
\end{prop}
\subsection{Analogy with the $q$-Farey tessellation}
In the remainder of this section, we will only work with the category $\C_2$, and with $q \in (0,1)$ (although everything we state is true for $q \in (1, \infty)$ up to sign/orientation).
We will use $M_{q}$ to denote
$$M_{q}:=m_{q}(\Stab(\C_{2})/\CC) \subset \P^{\bS}(\RR).$$ 
Our goal is to classify the points in the `boundary' $\dd\ol{M}_{q}:=\ol{M}_{q}\sm M_{q}$. 
We will show that the $\occ_{q}$ and $\ol{\hom}_{q}$ functionals lie in $\dd\ol{M}_{q}$, and then use \cref{qrz1} and \cref{qrz2} to completely describe $\dd\ol{M}_{q}$.
\par
In \cref{motivation}, we described `$q$-deformed rational curves' as limits of sequences of triangles in the $q$-deformed Farey tessellation. These curves satisfy two notable properties:
\begin{enumerate}
\item they are in natural, $B_{3}$-equivariant bijection with the rational numbers;
\item their union is a dense subset of the boundary of the $q$-Farey tessellation.
\end{enumerate}
\par
The same phenomenon occurs in the embedding of $\Stab(\C_{2})/\CC$ in projective space.
Let $\ol{\hom}_{q}(X),\occ_{q}(X)$ denote the functionals $\ol{\hom}_{q}(X,\cdotp), \occ_{q}(X,\cdotp)$ respectively, viewed as ellements of $\P^{\bS}(\RR)$. That is,
\[
\ol{\hom}_{q}(X):=[\ol{\hom}_{q}(X,Y)]_{Y \in \bS}, \quad \occ_{q}(X):=[\occ_{q}(X,Y)]_{Y \in \bS}.
\]
Set
\[I_{q,X} = [\ol{\hom}_{q}(X),\occ_{q}(X)]:=\{t\occ_{q}(X)+(1-t)\ol{\hom}_{q}(X)\mid t \in [0,1]\}\]
as indicated in the introduction to the paper.
We will show shortly that the convex sets $I_{q,X}$ for $X \in \bS$ arise as limits of stability conditions. The functional $\ol{\hom}_{q}(X)$ should be thought of as the correspondent of the left $q$-deformed rational number naturally associated to $X$, and $\occ_{q}(X)$ as the correspondent of the right $q$-deformed rational; together with their convex combinations, we obtain a perfect analogue of the $q$-deformed rational curve.
\par
Recall the definition of standard stability conditions from \cref{stddef}. The subset of the space of stability conditions that corresponds to the first two Farey triangles is given by the following definition:
\begin{mydef}\cite[\textsection 5.3]{bap.deo.lic:20}
A standard stability condition $\t$ is {type-A} if $P_{21}$ is semistable with respect to $\t$, and {type-B} if $P_{12}$ is semistable with respect to $\t$. A type-A (resp, type-B) standard stability condition is {strictly type-A} (resp. {strictly type-B}) if it is not both type-A and type-B.
\end{mydef}
Let $\La$ denote the set of standard stability conditions, let $\La_{A}$ (resp. $\La_{B}$) denote the set of type-A (resp. type-B) standard stability conditions.
\par
Think of $\La_{A}$ as the Farey triangle with vertices $0,\infty,1$, and $\La_{B}$ as the Farey triangle with vertices $0,\infty,-1$. So the initial edge $E$ is the intersection $\La_{A}\cap \La_{B}$. 
Its image under the $q$-mass map is the interval between $\occ_{q}(P_{1})$ and $\occ_{q}(P_{2})$:
\begin{lemma}\label{initialedge}
We have an equality:
$$m_{q}(\La_{A}\cap \La_{B})=\{t\occ_{q}(P_{1})+(1-t)\occ_{q}(P_{2})\mid t \in (0,1)\}.$$
\end{lemma}
We can now state how the `intervals' $I_{q,X}$
arise in $\dd\ol{M}_{q}$.
\begin{prop}\label{limits1cor}
Consider any $X \in \bS$, and suppose that $X=\b P_{1}[n]$, then
$$\lim_{m \to \infty}\ol{m_{q}(\b \s_{1}^{-m}(\La_{A}\cap \La_{B}))}=I_{q,X}.$$
\end{prop}
\begin{proof}
We first prove this statement for $\beta = \operatorname{id}$, and then deduce it for all other $\beta$ by equivariance. Thus we have to show that 
$$\lim_{m \to \infty}\ol{m_{q}(\s_{1}^{-m}(\La_{A}\cap \La_{B}))}= I_{q,P_1} = [\ol{\hom}_{q}(P_{1}),\occ_{q}(P_{1})].$$
 By \cref{initialedge},
$$\ol{m_{q}(\La_{A}\cap \La_{B})}=\{t\occ_{q}(P_{1})+(1-t)\occ_{q}(P_{2})\mid t \in [0,1]\}.$$
Since $\s_{1}^{-1}P_{1}=P_{1}[1]$, we have that
$$\lim_{m \to \infty}\s_{1}^{-m}\occ_{q}(P_{1})=\occ_{q}(P_{1}).$$
Thus, all we need to show is that
$$\lim_{m \to \infty}\s_{1}^{-m}\occ_{q}(P_{2})=\ol{\hom}_{q}(P_{1}).$$
\par
Let $X,Y \in \bS$ be such that $X \geq 0$, $Y \leq 0$. Consider the following abbreviated vector  in $\P^{\bS}(\RR)$, representing $\s_{1}^{-m}\occ_{q}(P_{1})$
\begin{align*}
\s_{1}^{-m}
\begin{bmatrix}
\occ_{q}(P_{2}, P_{1}) \\
\occ_{q}(P_{2},P_{2})\\
\occ_{q}(P_{2}, X)\\
\occ_{q}(P_{2},Y)
\end{bmatrix}
&=\begin{bmatrix}
\occ_{q}(P_{2}, \s_{1}^{m}P_{1}) \\
\occ_{q}(P_{2},\s_{1}^{m}P_{2})\\
\occ_{q}(P_{2}, \s_{1}^{m}X)\\
\occ_{q}(P_{2},\s_{1}^{m}Y)
\end{bmatrix}
\\ &=\begin{bmatrix}
q^{-m}\\
[m]_{q^{-1}}\\
[m-1]_{q^{-1}}(1-q^{-1})\occ_{q}(P_{2},X)+[m]_{q^{-1}}\occ_{q}(P_{1},X)\\
\occ_{q}(P_{2},Y)+(q^{-1}-1)[m]_{q^{-1}}\occ_{q}(P_{2},Y)+[m]_{q^{-1}}\occ_{q}(P_{1},Y)
\end{bmatrix}.
\end{align*}
We multiply all the coordinates by $\frac{q^{-1}}{[m]_{q^{-1}}}$ to obtain:
\begin{align*}
\begin{bmatrix}
\frac{q^{-m-1}}{[m]_{q^{-1}}}\\
q^{-1}\\
\frac{[m-1]_{q^{-1}}}{[m]_{q^{-1}}}(q^{-1}-q^{-2})\occ_{q}(P_{2},X)+q^{-1}\occ_{q}(P_{1},X)\\
\frac{q^{-1}}{[m]_{q^{-1}}}\occ_{q}(P_{2},Y)+(q^{-2}-q^{-1})\occ_{q}(P_{2},Y)+q^{-1}\occ_{q}(P_{1},Y)
\end{bmatrix}.
\end{align*}
When we send $m$ to infinity, we get:
\begin{align*}
\begin{bmatrix}
q^{-2}-q^{-1}\\
q^{-1}\\
(1-q^{-1})\occ_{q}(P_{2},X)+q^{-1}\occ_{q}(P_{1},X)\\
(q^{-2}-q^{-1})\occ_{q}(P_{2},Y)+q^{-1}\occ_{q}(P_{1},Y)
\end{bmatrix}.
\end{align*}
Comparing to the $\hom$-values lemma (\cref{homvalues}), we see that this is indeed equal to $\ol{\hom}_{q}(P_{1})$.
\end{proof}
We will check that the appearance of these functionals really is a non-trivial result of taking the closure.
\begin{lemma}[Disjointness 1]\label{disjoint1}
For any $X \in \bS$ and all $t \in [0,1]$, 
$$t\occ_{q}(X)+(1-t)\ol{\hom}_{q}(X) \notin M_{q}.$$
\end{lemma}
To prove this proposition, we need to give a more thorough account of the dynamics of the $q$-deformed triangle inequality. Consider any stability condition $\t \in \Stab(\C_{2})/\CC$. By the triangle inequality for the $q$-deformed mass \cite[Proposition 3.3]{ike:20}, we have the inequalities
\begin{align} m_{q,\t}(P_{21}) &\leq m_{q,\t}(P_{1}) +m_{q,\t}(P_{2}),\label{triangleinequality1}\\
m_{q,\t}(P_{1}) &\leq q^{-1}m_{q,\t}(P_{2})+m_{q,\t}(P_{21}),\label{triangleinequality2}\\
m_{q,\t}(P_{2}) &\leq m_{q,\t}(P_{21})+qm_{q,\t}(P_{1}). \label{triangleinequality3}
\end{align}
\par
The next lemma describes when these inequalities degenerate into equalities; it is the $q$-analogue of the phenomenon noted at \cite[Remark 5.14]{bap.deo.lic:20}.
\par
\begin{lemma}[$q$-Degeneracy]\label{qmassdegeneracy}
Consider any stability condition $\t$. If $\t$ is strictly type-A 
standard, then the triangle inequalities \eqref{triangleinequality1}, \eqref{triangleinequality2}, \eqref{triangleinequality3} 
are strict inequalities. If $\t$ is not strictly standard, then one of the triangle inequalities \eqref{triangleinequality1}, \eqref{triangleinequality2}, \eqref{triangleinequality3}
is an equality.
\end{lemma}
Interchanging the indices 1 and 2 gives corresponding inequalities for type-B stability conditions, for which \cref{qmassdegeneracy} holds.
\par
Now we are ready to prove the disjointness of $I_{q,X}$ and $M_{q}$.
\begin{proof}[Proof of \cref{disjoint1}]
Consider any $X \in \bS$. 
That $\occ_{q}(X) \notin M_{q}$ follows immediately from the fact that $\occ_{q}(X,X)=0$, but $m_{q,\t}(X) >0$ for all $\t \in \Stab(\C_{2})/\CC$. Nevertheless, in order to obtain disjointness for all the convex combinations,  we must give a more involved proof for both $\occ_{q}(X)$ and $\ol{\hom}_{q}(X)$.
\par
Suppose to generate a contradiction that $\occ_{q}(X), \ol{\hom}_{q}(X) \in M_{q}$. Since every stability condition is in the braid group orbit of $\La$, we can assume that 
\[\occ_{q}(X)=m_{q}(\t)\quad\text{and}\quad \ol{\hom}_{q}(X)=m_{q}(\t'),\]
for some $\t, \t' \in \La$.  
Let $X = \b P_{1}$ for some braid $\b \in B_{3}$ which can be written in strict continued form. 
Assume for convenience that $\b$ has form \eqref{continued1} in the notation of \cref{continued}, and therefore $\b^{-1}\s_{1}$ has form \eqref{continued2}, that is,
\begin{equation}\label{formbeta}
\b^{-1}\s_{1}=
\s_{1}^{a_{1}}\s_{2}^{-a_{2}}...\s_{1}^{a_{2n-1}}\s_{2}^{-a_{2n}}\s_{1}^{M}\o^{N}.\end{equation}
Thus by \cref{rzprop},
\begin{align*}
\ol{\b^{-1}\s_{1}}&=
\begin{bmatrix}
\occ_{q}(P_{2}, \b^{-1} P_{1}[-1]) &-q^{-1}\occ_{q}(P_{2}, \b^{-1}P_{12})\\
-q \occ_{q}(P_{1}, \b^{-1} P_{1}[-1]) &\occ_{q}(P_{1}, \b^{-1}P_{12})
\end{bmatrix}
\\ &= \begin{bmatrix}
\occ_{q}(P_{2}, \b^{-1} P_{1}) &-q^{-1}\occ_{q}(P_{2}, \b^{-1}P_{2})\\
-q \occ_{q}(P_{1}, \b^{-1} P_{1}) &\occ_{q}(P_{1}, \b^{-1}P_{2})
\end{bmatrix}
\begin{bmatrix}
q^{-1} & -q^{-1} \\ 0 & 1
\end{bmatrix}.
\end{align*}
Therefore, for $i \in \{1,2\}$,
\begin{equation}\label{calc1}
\occ_{q}(P_{i}, \b^{-1}P_{12})=\occ_{q}(P_{i}, \b^{-1}P_{1})+\occ_{q}(P_{i}, \b^{-1}P_{2}).
\end{equation}
Now consider the two cases $M=0$ and $M>0$, where $M$ is the exponent of $\s_{1}$ in \cref{formbeta}.
If $M>0$, then $\b \s_{1}^{-1}$ has form \eqref{continued2}, and a similar matrix calculation gives
\begin{equation}\label{calc2}
\occ_{q}(P_{i}, \b^{-1}P_{21})=\occ_{q}(P_{i}, \b^{-1}P_{2})-q\occ_{q}(P_{i}, \b^{-1}P_{1}).
\end{equation}
If $M=0$, then $\b \s_{2}$ has form \eqref{continued2}, and another similar matrix calculation gives
\begin{equation}\label{calc3}
\occ_{q}(P_{i}, \b^{-1}P_{21})=\occ_{q}(P_{i}, \b^{-1}P_{1})-q^{-1}\occ_{q}(P_{i}, \b^{-1}P_{2}).
\end{equation}
Assume $i=1$ (the statements for $i=2$ are needed for $\ol{\hom}_{q}$). Then \cref{calc1,calc2,calc3}, along with our assumption that $m_{q}(\t)=\occ_{q}(C)$, imply that:
\begin{equation}\label{homcontra0}
m_{q, \t}(P_{12})=m_{q,\t}(P_{1})+m_{q,\t}(P_{2}),
\end{equation}
and either
\begin{equation}\label{homcontra1}
m_{q,\t}(P_{2})=m_{q,\t}(P_{21})+qm_{q,\t}(P_{1}),
\end{equation}
or
\begin{equation}\label{homcontra2}
m_{q,\t}(P_{1})=q^{-1}m_{q,\t}(P_{2})+m_{q,\t}(P_{21}).
\end{equation}
Thus the degeneracy lemma (\cref{qmassdegeneracy}) implies that $\t$ is neither strictly type-A, nor strictly type-B. Since we assumed that $\t$ is standard, this implies that $\t$ is both type-A and type-B. However, any stability condition which is both type-A and type-B satisfies
\begin{equation}\label{homcontra3}
m_{q,\t}(P_{21})=m_{q,\t}(P_{1})+m_{q,\t}(P_{2}).
\end{equation}
If we were to combine \cref{homcontra3} with either \cref{homcontra1} or \cref{homcontra2}, we would get that $m_{q,\t}(P_{i})=0$ for either $i = 1$ or $i=2$, giving the desired contradiction (as the $q$-mass associated to any stability condition is strictly positive on all complexes). A contradiction can be generated in the cases that $\b$ has the form \eqref{continued2}, \eqref{continued3}, \eqref{continued4}, in the same way. This shows that $\occ_{q}(X) \notin M_{q}$.
\par
Recall our supposition that $\ol{\hom}_{q}=m_{q}(\t')$. By our computation of $\ol{\hom}_{q}$ in terms of $\occ_{q}$ (\cref{homvalues}), for $j \in \{1,2,12,21\}$ we have
\begin{equation}\label{homvalueeq}
\ol{\hom}_{q}(P_{1}, \b^{-1}P_{j})=(q^{-2}-q^{-1})\occ_{q}(P_{2}, \b^{-1}P_{j}) +q^{-1}\occ_{q}(P_{1}, \b^{-1}P_{j}).
\end{equation}
 If we combine \cref{homvalueeq} with \cref{calc1,calc2,calc3}, we re-derive \cref{homcontra0,homcontra1,homcontra2}, replacing $\t$ with $\t'$. This delivers the same contradiction as in the case of $\t$, proving that $\ol{\hom}_{q}(X) \notin M_{q}$.
\par
The proof for the rest of $I_{q,X}$ follows from the argument above. The contradiction we obtained from the degenerating inequalities carries through when we pass to convex combinations.
\end{proof}
\subsection{Embedding the boundary of $\ol{M_q}$ into $\Q_{q} \cup \mathbb{I}_{q}$}
We will now use \cref{qrz1,qrz2} (the Rouquier--Zimmermann Theorems) to define a homeomorphism (essentially a projection map) from the sets $I_{q,X} = [\ol{\hom}_{q}(X),\occ_{q}(X)]$, lying in $\P^{\bS}(\RR)$, onto the corresponding $q$-deformed curves from~\cref{subsec:topology-q-tessellation}.
\par
Let $\OO_{q}$ denote the set 
$$\OO_{q}:=\bigsqcup_{X \in \bS}I_{q,X} \subset \P^{\bS}(\RR).$$

Recall that $\Q_{q}$ denotes the union of all the $q$-deformed rational curves, together with their endpoints.
For $X = P_1$, fix a homeomorphism from the interval $I_{q,P_{1}} = [\ol{\hom}_q(P_1), \occ_q(P_1)]$ to the complete geodesic between $1/(1-q)$ and $\infty$, together with its endpoints.
Extend this by the $B_3$ action to a $B_{3}$-equivariant $\rho_{q}:\OO_{q} \to \Q_{q}$, so that $\rho_{q}$ maps $\ol{\hom}_{q}(X)$ and $\occ_{q}(X)$ onto $[r/s]_{q}^{\flat}$ and $[r/s]_{q}^{\sharp}$ respectively, where $r/s$ is the rational number corresponding to $X$. In particular, this implies that $\rho_{q}$ is a bijection.
\par
The closure of $\OO_q$ is contained in $\ol{M}_{q}$, so we need to incorporate limits of points in $\OO_{q}$ into our classification of $\ol{M}_{q}$. Let $\ol{\OO}_{q}$ denote the union
$$\ol{\OO}_{q}:=\OO_{q} \cup \left\{\lim_{n \to \infty}\occ_{q}(X_{n})\mid (X_{n})_{n \in \N} \in \bS^{\N}\right\}.$$
As the notation suggests, $\ol{\OO}_{q}$ is in fact the closure of $\OO_{q}$ (see \cref{closure} below).
We now extend $\rho_{q}$ to a $B_{3}$-equivariant homeomorphism between $\ol{\OO}_{q}$ and $\Q_{q} \cup \mathbb{I}_{q}$, using the following lemma. 
\begin{lemma}\label{homeomorphism}
A sequence $(\occ_{q}(X_{n}))_{n \in \N}$ converges in $\P^{\bS}(\RR)$ if and only if $(\rho_{q}(\occ_{q}(X_{n})))_{n \in \N}$ converges in $\Q_{q} \cup \mathbb{I}_{q}$. Two sequences $(\occ_{q}(X_{n}))_{n \in \N}$ and $(\occ_{q}(Y_{n}))_{n \in \N}$ converge to the same limit if and only if $(\rho_{q}(\occ_{q}(X_{n})))_{n \in \N}$ and $(\rho_{q}(\occ_{q}(Y_{n})))_{n \in \N}$ converge to the same limit.
\end{lemma}
\begin{proof}
For each statement, the forwards direction is trivial, by \cref{qrz1}. The backwards direction follows from the fact that, by \cref{rzprop}, for any $X, Y \in \bS$, we have
$$\occ_{q}(X,Y)=\occ_{q}(X,P_{1})\occ_{q}(P_{2}, Y)+\occ_{q}(X, P_{2})\occ_{q}(P_{1},Y).$$
So for any $(X_{n})_{n \in \N}$ and any $Y \in \bS$, if the sequence of fractions $\big(\frac{\occ_{q}(X,P_{2})}{\occ_{q}(X,P_{1})}\big)_{n \in \N}$ converges, then so does the sequence of points in projective space
$$\begin{bmatrix} \occ_{q}(X_{n}, P_{2}) \\ \occ_{q}(X_{n}, P_{1}) \\ \occ_{q}(X_{n}, Y) \end{bmatrix}=\begin{bmatrix} \occ_{q}(X_{n}, P_{2}) \\ \occ_{q}(X_{n}, P_{1}) \\ \occ_{q}(X_{n}, P_{1})\occ_{q}(P_{2},Y)+ \occ_{q}(X_{n}, P_{2})\occ_{q}(P_{1},Y) \end{bmatrix}.$$
\end{proof}
\begin{lemma}\label{closure}
The set $\ol{\OO}_{q}$, as defined above, is the closure of $\OO_{q}$.
\end{lemma}
\begin{proof}
Consider any sequence in $\OO_{q}$, with the form 
$$(t_{n}\occ_{q}(X_{n})+(1-t_{n})\ol{\hom}_{q}(X_{n}))_{n \in \N}.$$
Assuming that the sequence $(X_{n})_{n \in \N}$ contains infinitely many distinct elements of $\bS$, using elementary properties of limits of right $q$-deformed rational numbers, it can be shown that
$$\lim_{n \to \infty}\rho_{q}(\occ_{q}(X_{n}))=\lim_{n \to \infty}\rho_{q}(\ol{\hom}_{q}(X_{n})).$$
\par
On the other hand, by the proof of \cref{limits1cor}, there exists a sequence $(Y_{n})_{n \in \N} \in \bS^{\N}$ such that
$$\lim_{n \to \infty}\occ_{q}(Y_{n})=\lim_{n \to \infty}\ol{\hom}_{q}(X_{n}).$$
Thus,
$$\lim_{n \to \infty}\rho_{q}(\occ_{q}(X_{n}))=\lim_{n \to \infty}\rho_{q}(\occ_{q}(Y_{n})),$$
and so by \cref{homeomorphism},
$$\lim_{n \to \infty}\occ_{q}(X_{n})=\lim_{n \to \infty}\occ_{q}(Y_{n})=\lim_{n \to \infty}\ol{\hom}_{q}(X_{n}).$$
Thus, under the assumption that the sequence $(X_{n})_{n \in \N}$ contains infinitely many distinct elements of $\bS^{\N}$, we have
$$\lim_{n\to \infty}(t_{n}\occ_{q}(X_{n})+(1-t_{n})\ol{\hom}_{q}(X_{n}))_{n \in \N}=\lim_{n \to \infty}\occ_{q}(X_{n}).$$
\end{proof}
The new points we have picked up by taking limits of $\occ_{q}$ functionals are also non-trivial members of the closure $\ol{M}_{q}$.
\begin{lemma}[Disjointness 2]\label{disjoint2}
For any $p \in \ol{\OO}_{q}$, $p \notin M_{q}$.
\end{lemma}
\begin{proof}
Consider any $p \in \ol{\OO}_{q} \sm \OO_{q}$, and suppose to generate a contradiction that $p=m_{q}(\t)$ for some $\t \in \La$.
There exists a sequence $(X_{k})_{k \in \N} \in \bS^{\N}$ such that
$$\lim_{k \to \infty}\occ_{q}(X_{k})=p.$$
Let $X_{k}=\b_{k}P_{1}$, where $\b_{k}$ is in strict continued form. Passing to a subsequence if necessary, we can assume for convenience that $\b_{k}$ either has form \eqref{continued1} for all $k \in \N$ or form \eqref{continued2} for all $k \in \N$, in the notation of \cref{continued} (we can rule out the cases \eqref{continued3} and \eqref{continued4}, since in these cases the limit is trivial).  Let's assume that $\b_{k}$ has form \eqref{continued1} for all $k \in \N$. Then by the proof of the first disjointness lemma (\cref{disjoint1}) we have that for all $k \in \N$, $i \in \{1,2\}$,
\begin{equation}\label{calc12}
\occ_{q}(P_{1}, \b^{-1}_{k}P_{12})=\occ_{q}(P_{1}, \b^{-1}_{k}P_{1})+\occ_{q}(P_{1}, \b^{-1}_{k}P_{2}).
\end{equation}
And either
\begin{equation}\label{calc22}
\occ_{q}(P_{1}, \b^{-1}_{k}P_{21})=\occ_{q}(P_{1}, \b^{-1}_{k}P_{2})-q\occ_{q}(P_{1}, \b^{-1}_{k}P_{1}),
\end{equation}
or
\begin{equation}\label{calc32}
\occ_{q}(P_{1}, \b^{-1}_{k}P_{21})=\occ_{q}(P_{1}, \b^{-1}_{k}P_{1})-q^{-1}\occ_{q}(P_{1}, \b^{-1}_{k}P_{2}).
\end{equation}
Pass to a subsequence so that we have \cref{calc22} for all $k \in \N$, or \cref{calc32} for all $k \in \N$.
Taking the limit, we get that
\begin{equation*}
m_{q,\t}(P_{12})=m_{q,\t}(P_{1})+m_{q,\t}(P_{2}),
\end{equation*}
and either
\begin{equation*}
m_{q,\t}(P_{2})=m_{q,\t}(P_{21})+qm_{q,\t}(P_{1}),
\end{equation*}
or
\begin{equation*}
m_{q,\t}(P_{1})=q^{-1}m_{q,\t}(P_{2})+m_{q,\t}(P_{21}),
\end{equation*}
As in the proof of the disjointness lemma part 1, this implies a contradiction.
\end{proof}
\subsection{Classification of the compactification of $\C_{2}$}
The main theorem of this section is that we have now given a full description of the closure.
\begin{theorem}\label{boundary}
We have an equality 
$$\ol{M}_{q} =M_{q} \sqcup \ol{\OO}_{q}$$
in $\P^{\bS}(\RR)$.
\end{theorem}
We have a homeomorphism $\rho_{q}:\ol{\OO}_{q} \to \Q_{q} \cup \mathbb{I}_{q}$. So, this theorem implies that $\dd\ol{M}_{q}=\ol{\OO}_{q} \cong \RR\cup\{\infty\}$, that is, the boundary of the $q$-deformed Thurston compactification of $\Stab(\C_{2})$ is homeomorphic to a circle (see Part 3 of~\cref{qtopology}).
\begin{remark}
We stated without proof in \cref{qtopology} that $\Q_{q} \cup \mathbb{I}_{q}$ is equal to the boundary of the $q$-deformed Farey tesselation, so in fact $\rho_{q}$ maps $\ol{\mathbb{O}}_{q}$ onto $\dd \ol{L}_{q}$. Thus, $\rho_{q}$ is a $B_{3}$-equivariant homeomorphism from the boundary of the $q$-deformed compactification of $\Stab(\C_{2})/\CC$ onto the boundary of the $q$-deformed Farey tessellation.
\end{remark}
\cref{boundary} can be bootstrapped onto the proof of the corresponding \cite[Proposition 5.17]{bap.deo.lic:20} (which covers the case $q=1$), using `$q$-deformed Gromov coordinates'. Consider any standard stability condition $\t \in \Stab(\C_{2})/\CC$. By the triangle inequality for the $q$-deformed mass function \cite[Proposition 3.3]{ike:20}, there exist $a_{\t},b_{\t},c_{\t} \in \RR^{\geq 0}$ such that
$$m_{q,\t}(P_{1})=b_{\t}+c_{\t}, \; \;m_{q,\t}(P_{2})=a_{\t}+qc_{\t}, \;\; m_{q,\t}(P_{21})=a_{\t}+b_{\t}.$$
We call $(a_{\t}, b_{\t}, c_{\t})$ the $q$-Gromov coordinates associated to $\t$.
The following is the $q$-deformation of \cite[Proposition 5.14]{bap.deo.lic:20}. The proof is essentially identical.
\begin{prop}[$q$-Linearity]\label{qlinearity}
Consider any standard stability condition $\t \in \Stab(\C_{2})/\CC$, then
$$m_{q,\t}=a_{\t}\occ_{q}(P_{1})+b_{\t}\occ_{q}(P_{2})+c_{\t}\occ_{q}(P_{21}).$$
\end{prop}
\begin{proof}[Proof of \cref{boundary}]
We just need to show that 
$$M_{q} \cup \ol{\OO}_{q}$$
is closed in $\P^{\bS}(\RR)$ (since we have already shown that $\ol{\OO}_{q} \subset \dd \ol{M}_{q}$). Consider any sequence $(m_{q}(\t_{n}))_{n \in \N}$ which converges to a limit in $\P^{\bS}(\RR)$. Let $\t_{n}=\b_{n}\t_{n}^{\std}$, where each $\t_{n}^{\std}$ is a standard stability condition.
\par
Suppose $\{\b_{n}\mid n \in \N\}$ is a finite set. The proof that $\lim_{n  \to \infty}m_{q}(\t_{n})$ lies in $M_{q} \cup \ol{\OO}_{q}$ is easy, and identical to that in the proof of \cite[Proposition 5.17]{bap.deo.lic:20}.
Briefly, $\lim_{n  \to \infty}m_{q}(\t_{n})$ lies in the union of $\b_{n}(\ol{\La})$ for finitely many $\b_{n}$, and each of these sets lies in $M_{q} \cup \ol{\OO}_{q}$ by \cite[Proposition 5.15]{bap.deo.lic:20}.
\par
Now suppose that $\{\b_{n}\mid n \in \N\}$ is an infinite set.
Consider the two corresponding sets of sequences 
$$O_{1}=\big\{(\occ_{1}(\b_{n}P_{1}))_{n \in \N}, (\occ_{1}(\b_{n}P_{2}))_{n \in \N}, (\occ_{1}(\b_{n}P_{21}))_{n \in \N}\big\},$$
$$O_{q}=\big\{(\occ_{q}(\b_{n}P_{1}))_{n \in \N}, (\occ_{q}(\b_{n}P_{2}))_{n \in \N}, (\occ_{q}(\b_{n}P_{21}))_{n \in \N}\big\}.$$
Since $\ol{\OO}_{q}$ is sequentially compact, we can assume (simultaneously passing to subsequences if necessary) that all these sequences converge. Let us denote the limits $p_{1},p_{2},p_{21}$ and $p_{1,q}, p_{2,q}, p_{21,q}$ respectively. It is shown that $p_{1}=p_{2}=p_{21}$ in the proof of \cite[Proposition 5.17]{bap.deo.lic:20}. The limit $p_{1}=p_{2}=p_{21}$ corresponds to a real number, namely $\rho_{q=1}(p_{1})=\rho_{q=1}(p_{2})=\rho_{q=1}(p_{21})$. 
We split the remainder of the proof into two cases, depending on whether $\rho_{q=1}(p_{1})=\rho_{q=1}(p_{2})=\rho_{q=1}(p_{21})$ is rational or irrational.
\par
Suppose that $\rho_{q=1}(p_{1})=\rho_{q=1}(p_{2})=\rho_{q=1}(p_{21})$ is irrational. Then by the characterisation of limits of right $q$-rationals above (\cref{qlimits}),
we have that $\rho_{q}(p_{1,q})=\rho_{q}(p_{2,q})=\rho_{q}(p_{21,q})$ is the unique `$q$-irrational' corresponding to $\rho_{q}(p_{1})=\rho_{q}(p_{2})=\rho_{q}(p_{21})$. Since $\rho_{q}$ is a homeomorphism, this implies that $p_{1,q}=p_{2,q}=p_{21,q}$.

 By the $q$-linearity proposition (\cref{qlinearity}), we have sequences $(a_{n})_{n \in \N}$, $(b_{n})_{n \in \N}$, and $(c_{n})_{n \in \N}$ in $(\RR^{\geq 0})^{\N}$ such that
\begin{align*}
\lim_{n \to \infty} m_{q}(\t_{n})&=\lim_{n \to \infty} m_{q}(\b_{n}\t_{n}^{\std})
\\ &= \lim_{n \to \infty}\big(a_{n}\occ_{q}(\b_{n}P_{1})+b_{n}\occ_{q}(\b_{n}P_{2})+c_{n}\occ_{q}(\b_{n}P_{21})\big)
\\&= p_{1,q}=p_{2,q}=p_{21,q}.
\end{align*}
Since $p_{1,q}=p_{2,q}=p_{21,q} \in \ol{\OO}_{q}$, we have shown that the limit lies in $\ol{\OO}_{q}$.
\par
Suppose that $\rho_{q=1}(p_{1})=\rho_{q=1}(p_{2})=\rho_{q=1}(p_{21})$ is rational. By simultaneously passing to subsequences, we can assume that the image under $\rho_{q}$ of each of the sequences in $O_{1}$ is either an increasing sequence or a decreasing sequence. Let $X \in \bS$ correspond to the rational number $\rho_{q}(p_{1})=\rho_{q}(p_{2})=\rho_{q}(p_{21})$, so by \cref{qlimits}, each of $\rho_{q}(p_{1,q}), \rho_{q}(p_{2,q}), \rho_{q}(p_{21,q})$ is either the left or right $q$-deformation of $\rho_{q}(p_{1})$. Thus, since $\rho_{q}$ is a homeomorphism, $p_{1,q}, p_{2,q}, p_{21,q} \in \{ \occ_{q}(X), \ol{\hom}_{q}(X)\}$.
Again, by \cref{qlinearity}, we have sequences $(a_{n})_{n \in \N}, (b_{n})_{n \in \N}, (c_{n})_{n \in \N} \in (\RR^{\geq 0})^{\N}$ such that
$$\lim_{n \to \infty}m_{q}(\t_{n})=\lim_{n \to \infty}\big(a_{n}\occ_{q}(\b_{n}P_{1})+b_{n}\occ_{q}(\b_{n}P_{2})+c_{n}\occ_{q}(\b_{n}P_{21})\big).$$
Since we are in projective space, we can normalise so that for all $n \in \N$, $a_{n}+b_{n}+c_{n}=1$. Passing to subsequences if necessary, we can assume that each of $(a_{n})_{n \in \N}, (b_{n})_{n \in \N}, (c_{n})_{n\in \N}$ converges.
 Then since $\occ_{q}(\b_{n}P_{1}), \occ_{q}(\b_{n}P_{2}), \occ_{q}(\b_{n}P_{21})$ converge to $p_{1,q}, p_{2,q}, p_{21,q}$ respectively, we must have that
$$a_{n}\occ_{q}(\b_{n}P_{1})+b_{n}\occ_{q}(\b_{n}P_{2})+c_{n}\occ_{q}(\b_{n}P_{21})$$
converges to a convex combination of $p_{1,q}, p_{2,q}, p_{21,q}$. 
That is, there exists $t \in [0,1]$ such that
$$\lim_{n \to \infty}\big(a_{n}\occ_{q}(\b_{n}P_{1})+b_{n}\occ_{q}(\b_{n}P_{2})+c_{n}\occ_{q}(\b_{n}P_{21})\big)=t\ol{\hom}_{q}(X)+(1-t)\occ_{q}(X).$$
Thus we have shown that $\lim_{n \to \infty}m_{q}(\t_{n})$ lies in $[\ol{\hom}_{q}(X), \occ_{q}(X)]$.
This completes the proof.
\end{proof}

\appendix
\section{Combinatorial properties of left $q$-deformed rational numbers}\label{appendixa}

\subsection{Counting closures in quivers}
We return to thinking about right $q$-deformed rationals and left $q$-deformed rationals as formal rational functions in $q$. The following theorem of Morier-Genoud and Ovsienko shows that the coefficients of $\R^{\sharp}(q)$ and $\S^{\sharp}(q)$ have combinatorial significance.
For any oriented quiver $\G$, a \emph{closure} in $\G$ is a subquiver $\H$ such that there do not exist edges in $\G$ from vertices in $\H$ to vertices outside of $\H$. For $n \in \N$, an \emph{$n$-closure} is a closure with $n$ vertices. For any quiver $\G$, let $C_{n}(\G)$ denote the number of $n$-closures in $\G$. 
\par
Consider any rational number $r/s \in (1, \infty)$ with continued fraction expansion $[a_{1},...,a_{2n}]$. Let $\G_{r/s}^{\sharp}$ denote the following quiver.
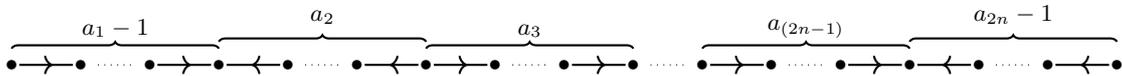
\begin{center}
  \begin{tikzpicture}[node distance=2.2em, font=\scriptsize]
    \tikzset {
      brace/.style={thick, decoration={brace}, decorate},
      rarrow/.style={-,
        thick,
        shorten <= -0.3em, shorten >= -0.3em,
        postaction=decorate,
        decoration={markings, mark= at position 0.7 with {\arrow{>}}}},
      larrow/.style={-,
        thick,
        shorten <= -0.3em, shorten >= -0.3em,
        postaction=decorate,
        decoration={markings, mark= at position 0.7 with {\arrow{<}}}}      
    }
    \node(a0) at (0,0) {$\bullet$};
    \node(m01) [right of=a0] {$\bullet$};
    \node(m02) [right of=m01] {$\bullet$};    
    \node(a1) [right of=m02] {$\bullet$};

    \node(m11) [right of=a1] {$\bullet$};
    \node(m12) [right of=m11] {$\bullet$};
    \node(a2) [right of=m12] {$\bullet$};

    \node(m21) [right of=a2] {$\bullet$};
    \node(m22) [right of=m21] {$\bullet$};
    \node(a3) [right of=m22] {$\bullet$};    

    \node(ax) [right of=a3] {$\bullet$};
    \node(mx1) [right of=ax] {$\bullet$};
    \node(mx2) [right of=mx1] {$\bullet$};    
    \node(ay) [right of=mx2] {$\bullet$};

    \node(my1) [right of=ay] {$\bullet$};
    \node(my2) [right of=my1] {$\bullet$};    
    \node(az) [right of=my2] {$\bullet$};    
    
    \draw[rarrow] (a0) -- (m01);
    \draw[dotted] (m01) -- (m02);
    \draw[rarrow] (m02) -- (a1);

    \draw[larrow] (a1) -- (m11);
    \draw[dotted] (m11) -- (m12);
    \draw[larrow] (m12) -- (a2);

    \draw[rarrow] (a2) -- (m21);
    \draw[dotted] (m21) -- (m22);
    \draw[rarrow] (m22) -- (a3);

    \draw[dotted] (a3) -- (ax);

    \draw[rarrow] (ax) -- (mx1);
    \draw[dotted] (mx1) -- (mx2);
    \draw[rarrow] (mx2) -- (ay);    

    \draw[larrow] (ay) -- (my1);
    \draw[dotted] (my1) -- (my2);
    \draw[larrow] (my2) -- (az);

    \draw[brace] (a0.north) -- node[above]{$a_1-1$} (a1.north);
    \draw[brace, decoration={raise=0.5ex}] (a1.north) -- node[above=1ex]{$a_2$} (a2.north);
    \draw[brace] (a2.north) -- node[above]{$a_3$} (a3.north);
    \draw[brace] (ax.north) -- node[above]{$a_{(2n-1)}$} (ay.north);
    \draw[brace, decoration={raise=0.5ex}] (ay.north) -- node[above=1ex]{$a_{2n}-1$} (az.north);    
  \end{tikzpicture}
\end{center}
Let $\wh{\G}_{r/s}^{\sharp}$ denote the following truncated copy of $\G_{r/s}$.
\begin{center}
  \begin{tikzpicture}[node distance=2.2em, font=\scriptsize]
    \tikzset {
      brace/.style={thick, decoration={brace}, decorate},
      rarrow/.style={-,
        thick,
        shorten <= -0.3em, shorten >= -0.3em,
        postaction=decorate,
        decoration={markings, mark= at position 0.7 with {\arrow{>}}}},
      larrow/.style={-,
        thick,
        shorten <= -0.3em, shorten >= -0.3em,
        postaction=decorate,
        decoration={markings, mark= at position 0.7 with {\arrow{<}}}}      
    }
    \node(a1) at (0,0) {$\bullet$};

    \node(m11) [right of=a1] {$\bullet$};
    \node(m12) [right of=m11] {$\bullet$};
    \node(a2) [right of=m12] {$\bullet$};

    \node(m21) [right of=a2] {$\bullet$};
    \node(m22) [right of=m21] {$\bullet$};
    \node(a3) [right of=m22] {$\bullet$};    

    \node(ax) [right of=a3] {$\bullet$};
    \node(mx1) [right of=ax] {$\bullet$};
    \node(mx2) [right of=mx1] {$\bullet$};    
    \node(ay) [right of=mx2] {$\bullet$};

    \node(my1) [right of=ay] {$\bullet$};
    \node(my2) [right of=my1] {$\bullet$};    
    \node(az) [right of=my2] {$\bullet$};    
    
    \draw[larrow] (a1) -- (m11);
    \draw[dotted] (m11) -- (m12);
    \draw[larrow] (m12) -- (a2);

    \draw[rarrow] (a2) -- (m21);
    \draw[dotted] (m21) -- (m22);
    \draw[rarrow] (m22) -- (a3);

    \draw[dotted] (a3) -- (ax);

    \draw[rarrow] (ax) -- (mx1);
    \draw[dotted] (mx1) -- (mx2);
    \draw[rarrow] (mx2) -- (ay);    

    \draw[larrow] (ay) -- (my1);
    \draw[dotted] (my1) -- (my2);
    \draw[larrow] (my2) -- (az);

    \draw[brace, decoration={raise=0.5ex}] (a1.north) -- node[above=1ex]{$a_2-1$} (a2.north);
    \draw[brace] (a2.north) -- node[above]{$a_3$} (a3.north);
    \draw[brace] (ax.north) -- node[above]{$a_{(2n-1)}$} (ay.north);
    \draw[brace, decoration={raise=0.5ex}] (ay.north) -- node[above=1ex]{$a_{2n}-1$} (az.north);    
  \end{tikzpicture}
\end{center}

\begin{theorem}[{\cite[Theorem 4]{mor.ovs:20}}]\label{mgothm}
Consider any rational number $r/s \in (1, \infty)$, with continued fraction expansion $[a_{1},...,a_{n}]$ and $q$-deformation
$$\Big[\frac{r}{s}\Big]_{q}^{\sharp}=\frac{\a^{\sharp}_{0}+\a_{1}^{\sharp}q+...+\a_{N-1}^{\sharp}q^{N-1}+\a_{N}^{\sharp}q^{N}}{\b_{0}^{\sharp}+\b_{1}^{\sharp}q+...+\b_{M-1}^{\sharp}q^{M-1}+\b_{M}^{\sharp}q^{M}}.$$ For all $i \in \N$,
\[
\a_{i}^{\sharp}=C_{i}\big(\G_{r/s}\big), \quad \b_{i}^{\sharp}=C_{i}\big(\wh{\G}_{r/s}\big).
\]
\end{theorem}
The following corollary (which may have independent interest) will be used later on. The left $q$-deformed rationals, which we motivated in \cref{sec:deformation} on geometric grounds, and in  \cref{sec:homological} on homological grounds, also admit a combinatorial interpretation directly analogous to that given for the right $q$-deformed rationals by \cref{mgothm}.  The left $q$-rational counts closures in a slightly augmented version of $\G_{r/s}^{\sharp}$. Let $\G^{\flat}_{r/s}$ denote the following quiver.

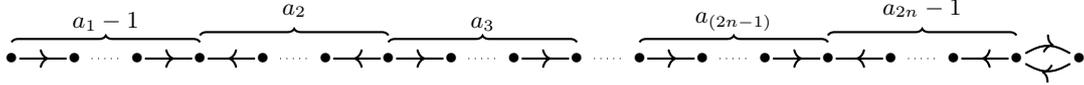
\begin{center}
  \begin{tikzpicture}[node distance=2em, font=\scriptsize]
    \tikzset {
      brace/.style={thick, decoration={brace}, decorate},
      rarrow/.style={-,
        thick,
        shorten <= -0.3em, shorten >= -0.3em,
        postaction=decorate,
        decoration={markings, mark= at position 0.7 with {\arrow{>}}}},
      larrow/.style={-,
        thick,
        shorten <= -0.3em, shorten >= -0.3em,
        postaction=decorate,
        decoration={markings, mark= at position 0.7 with {\arrow{<}}}}      
    }
    \node(a0) at (0,0) {$\bullet$};
    \node(m01) [right of=a0] {$\bullet$};
    \node(m02) [right of=m01] {$\bullet$};    
    \node(a1) [right of=m02] {$\bullet$};

    \node(m11) [right of=a1] {$\bullet$};
    \node(m12) [right of=m11] {$\bullet$};
    \node(a2) [right of=m12] {$\bullet$};

    \node(m21) [right of=a2] {$\bullet$};
    \node(m22) [right of=m21] {$\bullet$};
    \node(a3) [right of=m22] {$\bullet$};    

    \node(ax) [right of=a3] {$\bullet$};
    \node(mx1) [right of=ax] {$\bullet$};
    \node(mx2) [right of=mx1] {$\bullet$};    
    \node(ay) [right of=mx2] {$\bullet$};

    \node(my1) [right of=ay] {$\bullet$};
    \node(my2) [right of=my1] {$\bullet$};    
    \node(az) [right of=my2] {$\bullet$};

    \node(l) [right of=az] {$\bullet$};
    
    \draw[rarrow] (a0) -- (m01);
    \draw[dotted] (m01) -- (m02);
    \draw[rarrow] (m02) -- (a1);

    \draw[larrow] (a1) -- (m11);
    \draw[dotted] (m11) -- (m12);
    \draw[larrow] (m12) -- (a2);

    \draw[rarrow] (a2) -- (m21);
    \draw[dotted] (m21) -- (m22);
    \draw[rarrow] (m22) -- (a3);

    \draw[dotted] (a3) -- (ax);

    \draw[rarrow] (ax) -- (mx1);
    \draw[dotted] (mx1) -- (mx2);
    \draw[rarrow] (mx2) -- (ay);    

    \draw[larrow] (ay) -- (my1);
    \draw[dotted] (my1) -- (my2);
    \draw[larrow] (my2) -- (az);

    \path (az) edge[rarrow, bend left] (l)
    (l) edge [rarrow, bend left] (az);
    
    \draw[brace] (a0.north) -- node[above]{$a_1-1$} (a1.north);
    \draw[brace, decoration={raise=0.5ex}] (a1.north) -- node[above=1ex]{$a_2$} (a2.north);
    \draw[brace] (a2.north) -- node[above]{$a_3$} (a3.north);
    \draw[brace] (ax.north) -- node[above]{$a_{(2n-1)}$} (ay.north);
    \draw[brace, decoration={raise=0.5ex}] (ay.north) -- node[above=1ex]{$a_{2n}-1$} (az.north);    
  \end{tikzpicture}
\end{center}

And let $\wh{\G}^{\flat}_{r/s}$ denote the following truncation.

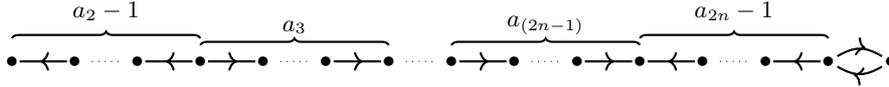
\begin{center}
  \begin{tikzpicture}[node distance=2em, font=\scriptsize]
    \tikzset {
      brace/.style={thick, decoration={brace}, decorate},
      rarrow/.style={-,
        thick,
        shorten <= -0.3em, shorten >= -0.3em,
        postaction=decorate,
        decoration={markings, mark= at position 0.7 with {\arrow{>}}}},
      larrow/.style={-,
        thick,
        shorten <= -0.3em, shorten >= -0.3em,
        postaction=decorate,
        decoration={markings, mark= at position 0.7 with {\arrow{<}}}}      
    }
    \node(a1) at (0,0) {$\bullet$};

    \node(m11) [right of=a1] {$\bullet$};
    \node(m12) [right of=m11] {$\bullet$};
    \node(a2) [right of=m12] {$\bullet$};

    \node(m21) [right of=a2] {$\bullet$};
    \node(m22) [right of=m21] {$\bullet$};
    \node(a3) [right of=m22] {$\bullet$};    

    \node(ax) [right of=a3] {$\bullet$};
    \node(mx1) [right of=ax] {$\bullet$};
    \node(mx2) [right of=mx1] {$\bullet$};    
    \node(ay) [right of=mx2] {$\bullet$};

    \node(my1) [right of=ay] {$\bullet$};
    \node(my2) [right of=my1] {$\bullet$};    
    \node(az) [right of=my2] {$\bullet$};

    \node(l) [right of=az] {$\bullet$};
    
    \draw[larrow] (a1) -- (m11);
    \draw[dotted] (m11) -- (m12);
    \draw[larrow] (m12) -- (a2);

    \draw[rarrow] (a2) -- (m21);
    \draw[dotted] (m21) -- (m22);
    \draw[rarrow] (m22) -- (a3);

    \draw[dotted] (a3) -- (ax);

    \draw[rarrow] (ax) -- (mx1);
    \draw[dotted] (mx1) -- (mx2);
    \draw[rarrow] (mx2) -- (ay);    

    \draw[larrow] (ay) -- (my1);
    \draw[dotted] (my1) -- (my2);
    \draw[larrow] (my2) -- (az);

    \path (az) edge[rarrow, bend left] (l)
    (l) edge[rarrow, bend left] (az);

    \draw[brace, decoration={raise=0.5ex}] (a1.north) -- node[above=1ex]{$a_2-1$} (a2.north);
    \draw[brace] (a2.north) -- node[above]{$a_3$} (a3.north);
    \draw[brace] (ax.north) -- node[above]{$a_{(2n-1)}$} (ay.north);
    \draw[brace, decoration={raise=0.5ex}] (ay.north) -- node[above=1ex]{$a_{2n}-1$} (az.north);    
  \end{tikzpicture}
\end{center}

\begin{cor}\label{leftcounting}

Consider any rational number $r/s \in (1, \infty)$, with left $q$-deformation
$$\Big[\frac{r}{s}\Big]_{q}^{\flat}=\frac{\a^{\flat}_{0}+\a^{\flat}_{1}q+...+\a^{\flat}_{N-1}q^{N-1}+\a^{\flat}_{N}q^{N}}{\b^{\flat}_{0}+\b^{\flat}_{1}q+...+\b^{\flat}_{M-1}q^{M-1}+\b^{\flat}_{M}q^{M}}.$$
Then for all $i \in \N$,
\begin{align*}
\a_{i}^{\flat}=C_{i}\big(\G_{r/s}^{\flat}\big), && \b_{i}^{\flat}=C_{i}\big(\wh{\G}_{r/s}^{\flat}\big).
\end{align*}
\end{cor}
We omit the proof, which is a straightforward application of \cref{mgothm} by way of the following formula for the left $q$-deformed rational (using the notation of \cref{qmatrices}):
$$\Big[\frac{r}{s}\Big]_{q}^{\flat} = \frac{q\R^{\sharp}(q)+\R^{\sharp\prime}(q)-q\R^{\sharp\prime}(q)}{q\S^{\sharp}(q)+\S^{\sharp\prime}(q)-q\S^{\sharp\prime}(q)}.$$
\subsection{Left $q$-deformed rational numbers and the Jones polynomial}
The next theorem links left $q$-deformed rationals to the Jones polynomial of rational (two-bridge) knots. The proof of the theorem uses \cref{qrz2}, although we suspect there exists a purely combinatorial proof.
\par
Due to work by Lee and Schiffler \cite[Theorem 1.2]{lee.sch:19} and by Morier-Genoud and Ovsienko \cite[Proposition A.3]{mor.ovs:20}, the proof of this theorem is not difficult, and does not involve any knot theory. For any rational number $r/s \in (1,\infty)$, let $V_{r/s}(q)$ denote the Jones polynomial associated with the class of rational knots parametrised by $r/s$, and let $\v V_{r/s}(q) \v$ denote the polynomial obtained by making each coefficient positive.
\begin{theorem}\label{jones}
Consider any rational number $r/s \in (1, \infty)$ with left $q$-deformation
$$\Big[\frac{r}{s}\Big]_{q}^{\flat}=\frac{\R^{\flat}(q)}{\S^{\flat}(q)}.$$
Then $\R^{\flat}(q)=\v V_{r/s}(q)\v$.
\end{theorem}
To prove this theorem, we will need the following lemma. For any rational number $r/s$, let $\H_{r/s}$ denote the following quiver (note that the orientation of the edges is opposite to that of $\G_{r/s}^{\flat}$).

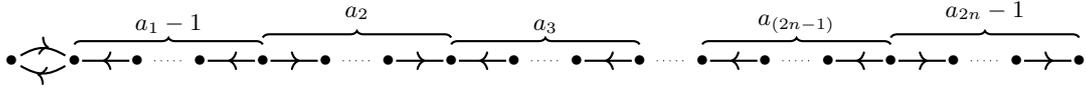
\begin{center}
  \begin{tikzpicture}[node distance=2em, font=\scriptsize]
    \tikzset {
      brace/.style={thick, decoration={brace}, decorate},
      rarrow/.style={-,
        thick,
        shorten <= -0.3em, shorten >= -0.3em,
        postaction=decorate,
        decoration={markings, mark= at position 0.7 with {\arrow{>}}}},
      larrow/.style={-,
        thick,
        shorten <= -0.3em, shorten >= -0.3em,
        postaction=decorate,
        decoration={markings, mark= at position 0.7 with {\arrow{<}}}}      
    }
    \node(a0) at (0,0) {$\bullet$};
    \node(m01) [right of=a0] {$\bullet$};
    \node(m02) [right of=m01] {$\bullet$};    
    \node(a1) [right of=m02] {$\bullet$};

    \node(m11) [right of=a1] {$\bullet$};
    \node(m12) [right of=m11] {$\bullet$};
    \node(a2) [right of=m12] {$\bullet$};

    \node(m21) [right of=a2] {$\bullet$};
    \node(m22) [right of=m21] {$\bullet$};
    \node(a3) [right of=m22] {$\bullet$};    

    \node(ax) [right of=a3] {$\bullet$};
    \node(mx1) [right of=ax] {$\bullet$};
    \node(mx2) [right of=mx1] {$\bullet$};    
    \node(ay) [right of=mx2] {$\bullet$};

    \node(my1) [right of=ay] {$\bullet$};
    \node(my2) [right of=my1] {$\bullet$};    
    \node(az) [right of=my2] {$\bullet$};

    \node(l) [left of=a0] {$\bullet$};
    
    \draw[larrow] (a0) -- (m01);
    \draw[dotted] (m01) -- (m02);
    \draw[larrow] (m02) -- (a1);

    \draw[rarrow] (a1) -- (m11);
    \draw[dotted] (m11) -- (m12);
    \draw[rarrow] (m12) -- (a2);

    \draw[larrow] (a2) -- (m21);
    \draw[dotted] (m21) -- (m22);
    \draw[larrow] (m22) -- (a3);

    \draw[dotted] (a3) -- (ax);

    \draw[larrow] (ax) -- (mx1);
    \draw[dotted] (mx1) -- (mx2);
    \draw[larrow] (mx2) -- (ay);    

    \draw[rarrow] (ay) -- (my1);
    \draw[dotted] (my1) -- (my2);
    \draw[rarrow] (my2) -- (az);

    \path (a0) edge[rarrow, bend left] (l)
    (l) edge [rarrow, bend left] (a0);
    
    \draw[brace] (a0.north) -- node[above]{$a_1-1$} (a1.north);
    \draw[brace, decoration={raise=0.5ex}] (a1.north) -- node[above=1ex]{$a_2$} (a2.north);
    \draw[brace] (a2.north) -- node[above]{$a_3$} (a3.north);
    \draw[brace] (ax.north) -- node[above]{$a_{(2n-1)}$} (ay.north);
    \draw[brace, decoration={raise=0.5ex}] (ay.north) -- node[above=1ex]{$a_{2n}-1$} (az.north);    
  \end{tikzpicture}
\end{center}
\begin{lemma}\label{jonescounting}
Consider any rational number $r/s \in (1, \infty)$. Let $\v V_{r/s}\v=\g_{0}+\g_{1}q+...+\g_{N-1}q^{N-1}+\g_{N}q^{N}$. Then for all $i \in \N$, $\g_{i}=C_{i}(\H_{r/s}).$
\end{lemma}
\cref{jonescounting} is an easy corollary of \cite[Proposition A.3]{mor.ovs:20}, since the sequence of coefficients of the normalised Jones polynomial $J_{r/s}(q)$ in their paper is just the reverse of the sequence of coefficients of the Jones polynomial $V_{r/s}$, and the number of $i$-closures of $\H_{r/s}$ is equal to the number of $(N-i)$-closures of $\H_{r/s}$ with opposite orientation.
\begin{proof}[Proof of \cref{jones}]
Let $[a_{1},...,a_{2n}]$ be the continued fraction expansion of $r/s$. Let $\ol{\R}^{\flat}(q)/\ol{\S}^{\flat}(q)$ be the left $q$-deformed rational associated to $[a_{2n},...,a_{1}]$. By \cref{leftcounting}, $\ol{\R}^{\flat}(q)$ counts the number of closures in $\H_{r/s}$ with its orientation reversed. It is easy to check that the number of $k$-closures in an $n$-vertex quiver $\G$ equals the number of $(n-k)$-closures in $\G$ with its orientation reversed. Thus, by \cref{jonescounting}, the sequence of coefficients of $\ol{\R}^{\flat}(q)$ is equal to the reverese of the sequence of coefficients of $\v V_{r/s}(q)\v$. So it is sufficient to show that the sequence of coefficients of $\R^{\flat}(q)$ is the reverse of the sequence of coefficients of $\ol{\R}^{\flat}(q)$.
\par 
By \cref{qrz2}, we have that 
$$\R^{\flat}(q)=\ol{\hom}_{q}(\s_{1}^{-a_{1}}\s_{2}^{a_{2}}...\s_{1}^{-a_{2n-1}}\s_{2}^{a_{2n}}P_{1}, P_{2}),$$
and that
\begin{align*}
\ol{\R}^{\flat}(q)&=\ol{\hom}_{q}(\s_{1}^{-a_{2n}}\s_{2}^{a_{2n-1}}...\s_{1}^{-a_{2}}\s_{2}^{a_{1}}P_{1}, P_{2})
\\&=\ol{\hom}_{q}(P_{1}, \s_{2}^{-a_{1}}\s_{1}^{a_{2}}...\s_{2}^{-a_{2n-1}}\s_{1}^{a_{2n}}P_{2})
\\ &=\ol{\hom}_{q}(P_{2}, \s_{1}^{-a_{1}}\s_{2}^{a_{2}}...\s_{1}^{-a_{2n-1}}\s_{2}^{a_{2n}}P_{1}).
\end{align*}
Let 
$$X=\s_{1}^{-a_{1}}\s_{2}^{a_{2}}...\s_{1}^{-a_{2n-1}}\s_{2}^{a_{2n}}P_{1}.$$
By the $\hom$-values lemma (\cref{homvalues}), we know that
$$\hom_{q}(P_{2},X)= (q^{-2}-q^{-1})\occ_{q}(P_{1},X)+q^{-1}\occ_{q}(P_{2},X),$$
and an analogous result for $\ol{\hom}_{q}(\cdotp,P_{2})$ gives that
$$\hom_{q}(X,P_{2})=(1-q^{-1})\occ_{q^{-1}}(P_{1},X)+q^{-1}\occ_{q^{-1}}(P_{2},X).$$
Thus, the sequence of coefficients of $\ol{\hom}_{q}(X, P_{2})$
is the reverse of the sequence of coefficients of 
$\ol{\hom}_{q}(P_{2}, X),$
completing the proof.
\end{proof}

\section{Bridgeland stability conditions and Harder--Narasimhan automata}\label{app:stability-and-automata}
\subsection{Background on Bridgeland stability conditions}
We assume the reader is familiar with the theory of Bridgeland stability conditions \cite{bri:07}.
To briefly recall, a stability condition on a triangulated category is specified as $\tau = (\cP_\tau, \cZ_\tau)$, where $\cP_\tau$ is a slicing and $\cZ_\tau$ is a compatible central charge.
An object of the category is called \emph{$\tau$-semistable} if it belongs to one of the full (abelian) subcategories $\cP_\tau(\phi)$ for some $\phi \in \mathbb{R}$.
In this case, the number $\phi$ is called the \emph{phase} of that semistable object.
An object is called \emph{$\tau$-stable} if it is semistable (of some phase $\phi$), and moreover if it is simple in the abelian category $\cP_\tau(\phi)$.

Recall also that every object of the category has a unique \emph{Harder--Narasimhan} filtration with respect to $\tau$.
\begin{definition}
  Fix a stability condition $\tau$ on a triangulated category.
  Let $X$ be any object of the category.
  The $\tau$-Harder--Narasimhan filtration of $X$ is the unique filtration 
  \[
    \begin{tikzcd}[column sep = 1em]
      0 \arrow{rr} && X_1 \arrow{rr} \arrow{dl} && X_2 \arrow{rr}\arrow{dl} && \cdots \arrow{rr} \arrow{dl}&&  X_n = X \arrow{dl}\\
      & A_1 \arrow[dashed]{ul} && A_2 \arrow[dashed]{ul} && A_3 \arrow[dashed]{ul} & \cdots & A_n \arrow[dashed]{ul}
    \end{tikzcd}
  \]
  with the following properties.
  \begin{enumerate}
  \item Each filtration factor $A_i$ is semistable of some phase $\phi_i$, and
  \item The numbers $\phi_i$ appear in decreasing order: $\phi_1 > \phi_2 \cdots > \phi_n$.
  \end{enumerate}
\end{definition}
Let $\mathcal{T}$ denote an arbitrary triangulated category. Recall that for any subset $I \subset \mathbb{R}$, we define $\cP_\tau(I)$ to be the full subcategory of $\mathcal{T}$ consisting of objects whose Harder--Narasimhan filtration factors have phases in $I$.
The subcategory $\cP_\tau([0,1))$ corresponding to the half-open interval $[0,1) \subset \mathbb{R}$ is a full abelian subcategory of $\mathcal T$, which is the heart of a bounded $t$-structure on $\mathcal T$ (see \cite[Section 3]{bri:07}).

Let $\Sigma_\tau$ denote the set of indecomposable semistable objects of $\mathcal T$ that lie in $\cP_\tau([0,1))$.
Then every indecomposable semistable object lies in $\Sigma_\tau$ up to an integer shift.
We can count the (graded) multiplicity of any object of $\mathcal T$ in a given object of $\Sigma_\tau$, as follows.
\begin{definition}\label{def:hn-multiplicity-vector}
  Fix some $q > 0$.
  The $\tau$-Harder--Narasimhan multiplicity vector (or simply the HN multiplicity vector) of a semistable object $A \in \mathcal T$ is an element of $\mathbb{Z}[q^{\pm}]^{\Sigma_\tau}$, defined as follows.
  Let $k$ be such that $A \in \cP_\tau([k,k+1))$.
  For any $B \in \Sigma_\tau$, let $n_{B}$ be the number of indecomposable summands of $A$ that are isomorphic to $B$.
  Then the coordinate of the HN multiplicity vector of $A$ at $B \in \Sigma_\tau$ is $n_{B} q^{k}$.
 
  The $\tau$-Harder--Narasimhan multiplicity vector of a general object of $\mathcal T$ is defined as the sum of the $\tau$-Harder--Narasimhan multiplicity vectors of its $\tau$-Harder--Narasimhan filtration factors.
\end{definition}
Next we introduce a graded version of Harder--Narasimhan mass (see \cite[Section 4.5]{dim.hai.kat.ea:14}, \cite[Section 3]{ike:20}).
\begin{definition}\label{massdef}
  Let $\tau$ be a stability condition on a triangulated category $\mathcal T$.
  Let $q$ be a fixed positive real number.
  Let $X$ be an object of $\mathcal T$ and consider its Harder--Narasimhan filtration:
  \[
    \begin{tikzcd}[column sep = 1em]
      0 \arrow{rr} && X_1 \arrow{rr} \arrow{dl} && X_2 \arrow{rr}\arrow{dl} && \cdots \arrow{rr} \arrow{dl}&&  X_n = X. \arrow{dl}\\
      & A_1 \arrow[dashed]{ul} && A_2 \arrow[dashed]{ul} && A_3 \arrow[dashed]{ul} & \cdots & A_n \arrow[dashed]{ul}
    \end{tikzcd}
  \]
  Let $\phi_i$ denote the phase of the semistable object $A_i$.
  The \emph{$q$-Harder--Narasimhan mass}, or \emph{$q$-mass} of the object $X$ is defined to be
  \[m_{q,\tau}(X) = \sum_{i = 1}^k q^{\phi_i}|Z_\tau(A_i)|.\]
\end{definition}
Note that for $q = 1$, the $q$-Harder--Narasimhan mass is simply the classical Harder--Narasimhan mass as defined in~\cite[Section 5]{bri:07}. The $q$-mass function is due to~\cite[Section 4.4]{dim.hai.kat.ea:14}.

Finally, we denote the space of stability conditions on a triangulated category $\mathcal T$ by $\Stab(\mathcal T)$.
There is a natural $\CC$-action on $\Stab(\mathcal T)$ defined as follows.

For $z = a + i\pi b \in \CC$ and a stability condition $\t=(\cP_{\t},\cZ_{\t})$, define $z\t=(\cP_{z\t}, \cZ_{z\t})$ by
\[\cP_{z\t}(\f) =\cP_{\t}(\f- b), \quad \cZ_{z\t}(X):=e^{z}\cZ_{\t}(X).\]
So for any $X \in \T$, the action of $z = a + i \pi b$ rotates the vector $Z(X)$ by the angle $\pi b$, scales it by the quantity $e^a$, and simultaneously shifts the slicing by the quantity $b$.
We will always consider the space $\Stab(\mathcal T)/\mathbb{C}$, which is the space of stability conditions up to this natural $\mathbb{C}$ action.

\subsection{Harder--Narasimhan and mass automata}\label{sec:automata}
Fix throughout this section a stability condition $\tau$ on $\mathcal T$.
Recall that $\Sigma_\tau$ is the set of indecomposable $\tau$-semistable objects of $\mathcal T$ that lie in $\cP_\tau([0,1))$.
For any $X \in \Ob \mathcal T$, set $\HN_\tau(X) \in \mathbb{Z}[q^{\pm}]^{\Sigma_\tau}$ to be the $\tau$-HN multiplicity vector of $X$.
Suppose also that a group $G$ acts on $\mathcal T$.

In this section we define Harder--Narasimhan automata on $\mathcal T$.
Our definition is just a graded version of the one from~\cite[Section 3]{bap.deo.lic:20}, but we recall the key elements here.
We begin with some notation.

Let $\Theta$ be an arbitrary quiver.
Let $\Theta_0$ and $\Theta_1$ denote the sets of vertices and edges of $\Theta$ respectively.
Given $\alpha \in \Theta_1$, we say that $s(\alpha)\in \Theta_0$ is its source, and $t(\alpha) \in \Theta_0$ is its target.

A $G$-labelling of $\Theta$ is a function $\Theta_1 \to G$.
Then every path (finite sequence of composable edges) in $\Theta$ automatically acquires a label, namely the product of the corresponding group elements in reverse order.

A $\Theta$-set $\underline{S}$ is a representation of $\Theta$ in the category of sets.
More precisely, it is specified by a collection of sets $\{S_v \mid v \in \Theta_0\}$, together with edge functions $\{f_\alpha \colon S_{s(\alpha)} \to S_{t(\alpha)} \mid \alpha \in \Theta_1\}$.

Let $A$ be any ring.
A $\Theta$-representation $\underline{M}$ of $A$-modules consists (as usual) of $A$-modules $\{M_v \mid v \in \Theta_0\}$, and $A$-linear maps $\{f_\alpha \colon M_{s(\alpha)} \to M_{t(\alpha)}\mid \alpha \in \Theta_1\}$.
Any $\Theta$-representation is also a $\Theta$-set.

A morphism of $\Theta$-sets (resp. $\Theta$-representations) consists of maps between the sets (resp. module maps between the modules) at each vertex, so that the square determined by every edge commutes.

A $G$-set $S$, together with a $G$-labelling on $\ell \colon \Theta_1 \to G$, naturally produces a $\Theta$-set, as follows.
We declare each $S_v$ to be $S$ for every $v \in \Theta_0$, and each $f_\alpha \colon S \to S$ to be the action of $\ell(\alpha)$ on $S$.
In particular, $\Ob \mathcal T$ is a $\Theta$-set, and we denote this as $\underline{\Ob \mathcal T}$.
Similarly, any $G$-representation, together with a $G$-labelling of $\Theta$, naturally produces a $\Theta$-representation.

\begin{definition}\label{def:hn-automaton}
  Fix some $q > 0$.
  A $\tau$-HN automaton for the $G$-action on $\mathcal T$ consists of the following data.
  \begin{enumerate}
  \item A $G$-labelled quiver $\Theta$ (and hence the $\Theta$-set $\underline{\Ob \mathcal T}$).
  \item A $\Theta$-subset $\underline{S} \subset \underline{\Ob \mathcal T}$.
  \item A $\Theta$-representation $\underline{M}$ of $\mathbb{Z}[q^\pm]$-modules.
  \item A morphism $i \colon \underline{S} \to \underline{M}$.
  \item A $\mathbb{Z}[q^\pm]$-linear map $h_v \colon M_v \to \mathbb{Z}[q^\pm]^{\Sigma_\tau}$ for every $v \in \Theta_0$.
  \end{enumerate}
  This data is required to satisfy
  \[\HN_\tau(s) = h_v(i_v(s))\]  
  for every $v \in \Theta_0$ and every $s \in S = S_v$.
\end{definition}

\begin{example}
  We illustrate the above definition with the following (somewhat trivial) example.
  Let $\mathcal T$ be any triangulated category, and fix a stability condition $\tau$ on it.
  Let $G$ be the group $\mathbb{Z}$, generated by the triangulated shift $(-)[1]$.
  Then an example of an HN automaton for this setup is given by the following data.
  \begin{enumerate}
  \item An infinite $\mathbb{Z}$-labelled quiver $\Theta$.
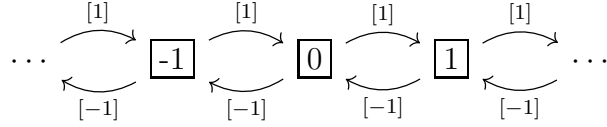
    \[
      \begin{tikzcd}
        \cdots\arrow[bend left]{r}{[1]}
        & \fbox{-1} \arrow[bend left]{r}{[1]}\arrow[bend left]{l}{[-1]}
        & \fbox{0} \arrow[bend left]{r}{[1]}\arrow[bend left]{l}{[-1]}
        & \fbox{1}  \arrow[bend left]{r}{[1]}\arrow[bend left]{l}{[-1]}
        & \cdots \arrow[bend left]{l}{[-1]}
      \end{tikzcd}
    \]
  \item The following $\Theta$-subset $\underline{S} \subset \underline{\Ob \mathcal T}$.
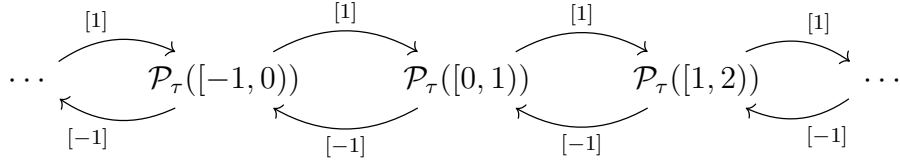
    \[
      \begin{tikzcd}
        \cdots\arrow[bend left]{r}{[1]}
        & \cP_\tau([-1,0)) \arrow[bend left]{r}{[1]}\arrow[bend left]{l}{[-1]}
        & \cP_\tau([0,1)) \arrow[bend left]{r}{[1]}\arrow[bend left]{l}{[-1]}
        & \cP_\tau([1,2))  \arrow[bend left]{r}{[1]}\arrow[bend left]{l}{[-1]}
        & \cdots \arrow[bend left]{l}{[-1]}
      \end{tikzcd}
    \]    
  \item The following $\Theta$-representation $\underline{M}$.
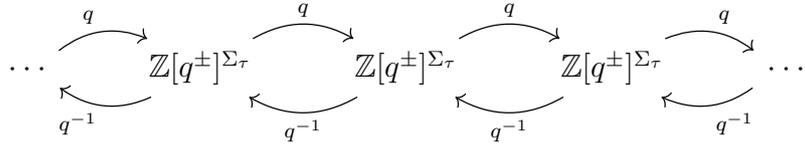
    \[
      \begin{tikzcd}
        \cdots\arrow[bend left]{r}{q}
        & \mathbb{Z}[q^\pm]^{\Sigma_\tau} \arrow[bend left]{r}{q}\arrow[bend left]{l}{q^{-1}}
        & \mathbb{Z}[q^\pm]^{\Sigma_\tau} \arrow[bend left]{r}{q}\arrow[bend left]{l}{q^{-1}}
        & \mathbb{Z}[q^\pm]^{\Sigma_\tau} \arrow[bend left]{r}{q}\arrow[bend left]{l}{q^{-1}}
        & \cdots \arrow[bend left]{l}{q^{-1}}
      \end{tikzcd}
    \]
  \item The morphism $i \colon \underline{S} \to \underline{M}$ is as follows.
    For any $j \in \mathbb{Z}$, the map $i_j \colon S_j \to M_j$ is defined as
    \[i_j(X) = \HN_\tau(X).\]
    With this definition, $i$ is $\Theta$-equivariant (i.e., a morphism) because
    \[\HN_\tau(X[1]) = q\cdot \HN_\tau(X).\]
    
  \item Finally, for each $j \in \mathbb{Z}$, the map $h_j \colon M_j \to \mathbb{Z}[q^\pm]^{\Sigma_\tau}$ is just the identity map.
    The compatibility condition
    \[\HN_\tau(X) = h_j(i_j(X))\]
    is obvious.
  \end{enumerate}
\end{example}

\begin{remark}
  The example above merely serves to illustrate the definition and does not compute anything interesting.
  However, appropriately engineered Harder--Narasimhan automata can be extremely useful for calculations, particularly for results about dynamics of autoequivalences in the category.
  An example of a computationally useful HN automaton is described in~\cref{subsec:c2-automaton} of this paper.
  Other examples can be found in~\cite{bap.deo.lic:20}.
\end{remark}

We also have the analogue of a \emph{mass automaton}, as described in~\cite[Section 3]{bap.deo.lic:20}.
\begin{definition}
Let $q$ be a fixed positive real number.
A $\tau$-mass automaton is analogous to a $\tau$-HN automaton, with the definition modified as follows.
The maps $h_v \colon M_v \to \mathbb{Z}[q^{\pm}]^{\Sigma_\tau}$ are replaced by linear maps $m_v \colon M_v \to \mathbb{R}$, satisfying
\[m_{q,\tau}(s) = m_v(i_v(s)).\]
\end{definition}
In particular, given an HN automaton for a stability condition $\tau$ and a fixed $q > 0$, we can produce a $\tau$-mass automaton by composing the HN vector with the $q$-mass function.

\bibliographystyle{alpha}
\end{document}